\documentclass[11pt,reqno]{amsart}
\usepackage{geometry}
\usepackage{stmaryrd}

\usepackage{esint}
\usepackage[svgnames]{xcolor} 
\usepackage[colorlinks,citecolor=red,pagebackref,hypertexnames=false,breaklinks]{hyperref}
\usepackage{pgf,tikz}
\usepackage{pdfsync}
\makeatletter
\renewcommand\theequation%
{\thesection.\arabic{equation}}
\@addtoreset{equation}{section}
\makeatother

\usepackage{amsfonts,amsmath,amssymb}
\usepackage{mathabx}
\usepackage{mathrsfs}
\usepackage[latin1]{inputenc}
\usepackage[usenames,dvipsnames]{pstricks}
\newtheorem{theorem}{Theorem}[section]

\newtheorem{remark}[theorem]{Remark}

\newtheorem{lemma}[theorem]{Lemma}

\newcommand{\RomanNumeralCaps}[1]{\MakeUppercase{\romannumeral #1}}
\newcommand{\R}{\mathbb{R}}
\newcommand{\C}{\mathbb{C}}

\newcommand{\N}{\mathbb{N}}

\newcommand{\g}{\mathrm{g}}

\newcommand{\s}{\mathbb{S}}

\newcommand{\p}{\partial}

\newcommand{\norm}[1]{\|#1\|}
\newcommand{\abs}[1]{|#1|}

\newcommand{\seq}[1]{\langle#1\rangle}

\newcommand{\To}{\longrightarrow}
\newcommand{\tild}{\widetilde}

\usepackage{graphicx}
\newcommand{\dive}{\textrm{div}}
\newcommand{\curl}{\textrm{curl}}
\usepackage{mathrsfs}
\usepackage[latin1]{inputenc}
\usepackage[T1]{fontenc}
\usepackage{color,xspace}
\usepackage[usenames,dvipsnames]{pstricks}
\usepackage{times}
\allowdisplaybreaks
\newcommand{\LL}{\mathcal{L}}

\newcommand{\Hs}{H_{\mathrm{scl}}}

\newcommand{\les}{{\,\lesssim\,}}
\newcommand{\ama}{\alpha}

\usepackage{subcaption}

\begin{document}

\title[Inverse problem for the Polyharmonic operator]{Stable determination of a second order perturbation of the polyharmonic operator by boundary measurements}
\author[N.~Aroua]{Nesrine Aroua}
\author[M.~Bellassoued]{Mourad Bellassoued}
\address{N.~Aroua. Universit\'e de Tunis El Manar, Ecole Nationale d'ing\'enieurs de Tunis, ENIT-LAMSIN, B.P. 37, 1002 Tunis, Tunisia}
\email{nesrine.aroua@enit.utm.tn }

\address{M.~Bellassoued. Universit\'e de Tunis El Manar, Ecole Nationale d'ing\'enieurs de Tunis, ENIT-LAMSIN, B.P. 37, 1002 Tunis, Tunisia}
\email{mourad.bellassoued@enit.utm.tn }

\date{\today}

\begin{abstract}
In this paper, we consider the inverse boundary value problem for the polyharmonic operator. We prove that the second order perturbations are uniquely determined by the corresponding Dirichlet to Neumann map. More precisely, we show in dimension $n \geq 3$, a logarithmic type stability estimate for the inverse problem under consideration.  
\end{abstract}
\maketitle
\section{Introduction and main results}
\subsection{Introduction}
The focus of this paper is the determination of a second order perturbation of the polyharmonic operator from the Dirichlet to Neumann map. We start with some notations. Let $\Omega\subset\R^n$, $n \geq 3$, be a bounded simply connected domain of $\R^n$ with $C^\infty$ boundary $\Gamma=\p\Omega$ and let $m \geq 2$. Given a symmetric matrix $A=(A_{jk})$, a vector field $B=(B_j)$ and a potential $q$, we consider the perturbed polyharmonic operator formally given by
\begin{align}\label{1.1}
\LL_{A,B,q}(x,D)&= (-\Delta)^m+A(x)D\cdot D+B(x)\cdot D+q(x)\cr
&=(-\Delta)^m+\sum_{j,k=1}^n A_{jk}(x)D_j D_k +\sum_{j=1}^n B_j(x)D_j+q(x),
\end{align}
where $D=i^{-1}\nabla$, $A\in W^{2,\infty}(\Omega,\C^{n^2})$, $B\in W^{1,\infty}(\Omega,\C^n)$ and $q\in L^\infty(\Omega,\C)$ are the perturbation coefficients. The operator $\LL_{A,B,q}(x,D)$ equipped with domain 
\begin{equation}\label{1.2}
\mathcal{D}(\LL_{A,B,q})=\{u\in H^{2m}(\Omega)\,:\, \gamma u:=\bigr(u,(-\Delta) u, \ldots, (-\Delta)^{m-1} u\bigr)|_{\Gamma}=0 \},
\end{equation}
is an unbounded closed operator on $L^2(\Omega)$ with purely discrete spectrum, see \cite{GG}.\\
For $m \geq 2$ and $r >0$, we denote by $\mathcal{H}^{m,r}(\Gamma)$ the product space
\begin{align*}
  \mathcal{H}^{m,r}(\Gamma)=\prod_{j=0}^{m-1}H^{2m-2j-r}(\Gamma ),  
\end{align*}
equipped with norm
\begin{align*}
 \|f\|_{ \mathcal{H}^{m,r}(\Gamma)}:=\sum_{j=0}^{m-1}\|f_j\|_{ {H}^{2m-2j-r}(\Gamma)},\quad f=(f_0, \ldots, f_{m-1}) \in \mathcal{H}^{m,r}(\Gamma).
\end{align*}
Assuming $0$ is not an eigenvalue of $\LL_{A,B,q}(x,D): \mathcal{D}(\LL_{A,B,q}) \rightarrow L^2(\Omega)$ and consider the boundary value problem with Navier boundary conditions
\begin{equation}\label{1.3}
\left\{\begin{array}{ll}
\LL_{A,B,q}(x,D)u =0 & \mathrm{in}\, \Omega,\cr
\gamma u  =f  &\mathrm{on}\, \Gamma,
\end{array}
\right.
\end{equation}
where $f=(f_0, \ldots, f_{m-1})\in \mathcal{H}^{m,\frac{1}{2}}(\Gamma)$. Then \eqref{1.3} has a unique solution $u:=u_f\in H^{2m}(\Omega)$ (see Lemma \ref{L0.1} below) and the boundary measurements are given by the Dirichlet to Neumann map (D-to-N), defined formally by
\begin{equation}\label{1.4}
\Lambda_{A,B,q}: f\mapsto  \widetilde{\gamma} u:=(\p_\nu u,\p_\nu(-\Delta) u, \ldots, \p_\nu(-\Delta)^{m-1} u)|_\Gamma,
\end{equation}
where $u\in H^{2m}(\Omega)$ satisfies the boundary value problem (\ref{1.3}). Here $\nu$ denotes the unit outer normal vector to the boundary $\Gamma$ at $x$. 
\medskip

Before stating our main result, we recall the following Lemma on the existence and uniqueness of a solution to the problem \eqref{1.3}, the proof is given by Lemma 5.13 in \cite{TH}.
\begin{lemma}\label{L0.1}
Let $A \in W^{2, \infty}(\Omega, \C^{n^2})$, $B \in W^{1, \infty}(\Omega, \C^n)$ and $q \in L^\infty(\Omega, \C)$. Suppose that $0$ is not an eigenvalue of $\LL_{A,B,q}(x, D)$ and $f \in \mathcal{H}^{m,\frac{1}{2}}(\Gamma)$. Then, there exists a unique solution $u \in H^{2m}(\Omega)$ to \eqref{1.3} satisfying
\begin{align*}
    \|u\|_{H^{2m}(\Omega)} \leq C \|f\|_{\mathcal{H}^{m,\frac{1}{2}}(\Gamma)}.
\end{align*}
Furthermore, we have $\widetilde{\gamma} u:=(\p_\nu u,\p_\nu(-\Delta) u, \ldots, \p_\nu(-\Delta)^{m-1} u)|_\Gamma \in \mathcal{H}^{m,\frac{3}{2}}(\Gamma) $ and there exists a constant $C > 0$ such that 
\begin{align*}
    \|\widetilde{\gamma} u\|_{\mathcal{H}^{m,\frac{3}{2}}(\Gamma)} \leq C \|f\|_{\mathcal{H}^{m,\frac{1}{2}}(\Gamma)}. 
\end{align*}
\end{lemma}
As a corollary of above Lemma, the Dirichlet-to-Neumann map $\Lambda_{A,B,q}$ given by \eqref{1.4}
is bounded from $\mathcal{H}^{m,\frac{1}{2}}(\Gamma)$ to
$\mathcal{H}^{m,\frac{3}{2}}(\Gamma)$. We denote by $\|\Lambda_{A,B,q}\|$ its norm.

\medskip
In this paper, we are interested in recovering the coefficients $A$, $B$ and $q$ from the boundary measurements unclosed in the Dirichlet to Neumann map
\begin{equation}\label{1.5}
\Lambda_{A,B,q}: \mathcal{H}^{m,\frac{1}{2}}(\Gamma)\To \mathcal{H}^{m,\frac{3}{2}}(\Gamma).
\end{equation}

In practice, the polyharmonic operators $(-\Delta)^m$, $m \geq2$, arise in many areas of physics and geometry, including the study of vibration of beams, the Kirchhoff-Love plate equation in the theory elasticity (for $m=2$), the continuum mechanics of buckling problems and the study of the Paneitz-Branson operator in conformal geometry, for more details we refer to \cite{GGS} and \cite{MVV}.\\
The uniqueness question of determination the lower order perturbations of polyharmonic operator $(-\Delta)^m$, $m \geq2$, was begun by Krupchyk, Lassas and Uhlmann who showed in \cite{K1} the uniqueness of the zeroth and the first order perturbations from the knowledge of D-to-N map. Later, they proved in \cite{K2} that the unique recovery of a first order perturbation of the biharmonic operator, i.e., $m=2$, is possible when the D-to-N map given only on a part of the boundary $\Gamma$, where $\Omega$ is a bounded domain.  In case where $\Omega$ is an unbounded domain, we refer the work of Yang \cite{Y}. In \cite{Gk}, Ghosh and Krishnan considered the higher order perturbation of polyharmonic operators with partial data on the boundary, where the coefficients attached to perturbed terms are isotropic. Later, Bhattacharyya and Ghosh established in \cite{B2} the uniqueness of a second order perturbations of a polyharmonic operator $(-\Delta)^m$, $m \geq 2$, from the D-to-N data. More precisely, they showed that, for $m>2$, the unique determination of a $2-$tensor field $A$, a vector field $B$ and a potential $q$ provided $A \in W^{3, \infty}(\Omega)$, $B \in W^{2, \infty}(\Omega)$ and $q \in L^\infty(\Omega)$ with some restrictions on the support of the coefficients and for $m=2$, they assumed that $A$ is an isotropic matrix. In the case of an anisotropic matrix, for $m=2$, they showed in \cite{B1} that the unique determination of second order perturbations of a biharmonic operator is possible from the boundary D-to-N map. Note that in the case of a lower order perturbation up to order $3$ of the biharmonic operator, there is an obstruction to uniqueness, as noted in \cite{B1}. Recently, Bhattacharyya, Krishnan and Sahoo established  in \cite{BKS} the unique recovery of a lower order anisotropic tensor perturbations up to order $m$ of a polyharmonic operator from measurements taken on a part of the boundary. All the above mentioned works are stated in the case where all the coefficients of the perturbed polyharmonic operator are sufficiently smooth. In case of lower regularity of the coefficients, we refer to \cite{A1}, \cite{AI}, \cite{KU} and to recent work \cite{BG}.
\smallskip

For the stability question, Choudhury and Krishnan proved in \cite{C1} a logarithmic type stability estimate for the zeroth order perturbation of the biharmonic operator, $(-\Delta)^2$, for the case when the Neumann data is measured on the whole boundary and a log-log type stability estimate for the case when the Neumann data is measured only on slightly more than half of the boundary. Later, Choudhury and Heck established in \cite{CH} a logarithmic type stability estimate for the zeroth order perturbation of the biharmonic operator with partial data when the inaccessible part of the domain is flat. A natural question to ask is whether higher order perturbations of the polyharmonic operator can be stably recovered from the knowledge of D-to-N map. In this work, we show that if we consider a second order perturbation of the polyharmonic operator  $(-\Delta)^m$, $m \geq 2$, of the form \eqref{1.1} then the coefficients $(A, B, q)$ depend stably on the bounded map $\Lambda_{A,B,q}$.

\subsection{Main stability results} We here state the main results of this paper concerning conditional logarithmic stability reconstruction of the second order perturbation coefficients $(A, B, q)$ from knowledge of the Dirichlet-to-Neumann map $\Lambda_{A,B,q}$ given by $\eqref{1.5}$.
\smallskip

Let us first indicate the required conditions for admissible coefficients $(A, B, q)$. In what follows, we assume that the second order perturbation $A$ takes the form
\begin{align}\label{cond_A}
A(x)=a(x) \operatorname{id},\quad \textrm{if} ~~m=2.
\end{align}
Let $M >0$ and $\sigma_1 > \frac{n}{2}+3$ be given. We define the class of admissible symmetric tensor $(A_{jk})$, $ \mathcal{A}_{\sigma_1}(M)$, by
\begin{align*}
    \mathcal{A}_{\sigma_1}(M)&=\{A \in W^{\operatorname{min}(2m, 6), \infty}(\Omega,\C^{n^2});~~ \|A\|_{H^{\sigma_1}(\Omega)} \leq M,~\textrm{and satisfying \eqref{cond_A}} \},\\
\end{align*}
Given $M >0$ and $\sigma_2 >\frac{n}{2}+1$, we define the class of admissible vector field $B$ and electric potential $q$ respectively by
\begin{align*}
    \mathcal{B}_{\sigma_2}(M)&=\{B \in W^{4,\infty}(\Omega,\C^n),~~ \|B\|_{H^{\sigma_2}(\Omega)} \leq M \},
\end{align*}
and 
\begin{align*}
   \mathcal{Q}(M)&=\{q \in L^{\infty}(\Omega),~~ \|q\|_{L^{\infty}(\Omega)} \leq M \}.
\end{align*}
Let us define, where $\mu >0$, the function $\Phi_\mu$ as follows 
$$
    \Phi_\mu(t)=\left\{
    \begin{array}{lll}
0, & t=0, \\
\abs{\ln t}^{-\mu }+t^{\frac{1}{2}}, & t >0,
    \end{array}
    \right.
$$
and we denote by $\mathcal{E}^\prime(\Omega)$ the space of all compactly supported distributions in $\Omega$.
\medskip

Our  main results state that it is possible to stably determine the symmetric tensor $A$, the vector field $B$ and the potential $q$, given the knowledge of the Dirichlet-to-Neumann map. Precisely, we will prove the following three theorems in sections \ref{section4} and \ref{section5}.

\begin{theorem}
Let $M >0$ and $\sigma_j >0$, $j=1,2,$ as above. Then, there exists a constant $C >0$ and $\mu_1, \kappa_1 \in (0,1)$, such that for any $A^{(j)}\in \mathcal{A}_{\sigma_1}(M) \cap \mathcal{E}^\prime(\Omega)$, $B^{(j)}\in\mathcal{B}_{\sigma_2}(M)\cap \mathcal{E}^\prime(\Omega)$ and $q^{(j)}\in \mathcal{Q}_{}(M)$, $j=1,2$, we have 
\begin{align}\label{1.8}
  \|A^{(1)}-A^{(2)}\|_{L^{\infty}(\Omega)} \leq C \Phi_{\mu_1}( \|\Lambda_{A^{(1)},B^{(1)},q^{(1)}}-\Lambda_{A^{(2)},B^{(2)},q^{(2)}}\|)^{\kappa_1}.
\end{align}
Here $C$ depends only on $\Omega$, $\sigma_1$, $\sigma_2$, $n$, $m$ and $M$, and $\mu_1$ depends only on $n$ and $m$.
\label{T1.2}
\end{theorem}

\begin{theorem}
Let $M >0$ and $\sigma_j >0$, $j=1,2,$ as above. Then, there exists a constant $C >0$ and $\mu_2, \kappa_2 \in (0,1)$, such that for any $A^{(j)}\in \mathcal{A}_{\sigma_1}(M)\cap \mathcal{E}^\prime(\Omega)$, $B^{(j)}\in\mathcal{B}_{\sigma_2}(M)\cap \mathcal{E}^\prime(\Omega)$ and $q^{(j)}\in \mathcal{Q}_{}(M)$, $j=1,2$, we have 
\begin{align}\label{1.9}
\|B^{(1)}-B^{(2)}\|_{L^{\infty}(\Omega)} \leq C  \Phi_{\mu_2}( \|\Lambda_{A^{(1)},B^{(1)},q^{(1)}}-\Lambda_{A^{(2)},B^{(2)},q^{(2)}}\|)^{\kappa_2}.
\end{align}
Here $C$ depends only on $\Omega$, $\sigma_1$, $\sigma_2$, $n$, $m$ and $M$, and $\mu_2$ depends only on $n$ and $m$.
\label{T1.3}
\end{theorem}

 \begin{theorem}
Let $M >0$ and $\sigma_j >0$, $j=1,2,$ as above. Then, there exists a constant $C >0$ and $\mu_3 \in (0,1)$, such that for any $A^{(j)}\in \mathcal{A}_{\sigma_1}(M)\cap \mathcal{E}^\prime(\Omega)$, $B^{(j)}\in\mathcal{B}_{\sigma_2}(M)\cap \mathcal{E}^\prime(\Omega)$ and $q^{(j)}\in \mathcal{Q}_{}(M)$, $j=1,2$, we have 
\begin{align}\label{1.10}
  \|q^{(1)}-q^{(2)}\|_{L^{\infty}(\Omega)} \leq C  \Phi_{\mu_3}( \|\Lambda_{A^{(1)},B^{(1)},q^{(1)}}-\Lambda_{A^{(2)},B^{(2)},q^{(2)}}\|).
\end{align}
Here $C$ depends only on $\Omega$, $\sigma_1$, $\sigma_2$, $n$, $m$ and $M$, and $\mu_3 \in (0,1)$ depends only on $n$ and $m$.
\label{T1.4}
\end{theorem}

Let us explain the main difficulties and ideas in the proof of
Theorems \ref{T1.2}, \ref{T1.3} and \ref{T1.4}. We start by
realling the qualitative argument due to Bhattacharyya and Ghosh
\cite{B2}. Their starting point
is the following orthogonal identity (see Lemma \ref{L3.1} below)
\begin{equation}
\int_\Omega (A(x)D \cdot D u+B(x)\cdot Du+q(x)u)\overline{u^*}dx=0,\label{*}
\end{equation}
which holds for $u$ and $u^*$ solving $\LL_{A^{(2)}, B^{(2)}, q^{(2)}}u=0$ and $\LL^*_{A^{(1)}, B^{(1)}, q^{(1)}}u^*=0$  respectively, wherever $\Lambda_{A^{(1)}, B^{(1)},
q^{(1)}}=\Lambda_{A^{(2)}, B^{(2)}, q^{(2)}}$. They, then proceed by constructing of special solutions, called complex geometric optics
solutions (C.G.O- solutions), that are to be used with the integral identity \eqref{*}. We mention that the method of C.G.O- solutions that is used
for proving uniqueness question for higher order elliptic operators
goes back to Sylvester and Uhlmann \cite{SU}. The complex geometric optics solutions to the polyharmonic equation take the special form 
\begin{align}\label{**}
    u(x,h)=e^{\frac{x \cdot \varrho}{h}}(\ama_1(x)+h \ama_2(x)+r(x,h)),
\end{align}
where $\varrho \in \mathbb{C}^n$, $h$ is a small parameter, $\ama_j$ $j=1,2,$ are complex amplitudes satisfying some transport equations and $r(\cdot, h)$ is remainder term that vanishes when $h$ goes to zero.\medskip\\
Now, any symmetric tensor $A \in H^k(\Omega)$ can be uniquely represented as the sum 
\begin{align}\label{***}
 A=\widetilde{A} +\nabla_{\text{sym}}V ,    
\end{align}
where the covector field $V \in  H^{k+1}(\Omega, \mathbb{C}^{n})$ satisfying the boundary condition
\begin{align}\label{***V}
 V=0,\quad \textrm{on}~~\Gamma,   
\end{align}
and the tensor field $\widetilde{A} \in  H^k(\Omega, \mathbb{C}^{n^2})$ satisfies the conditions
\begin{align}\label{***F}
  \sum_{j=1}^n \partial_j \widetilde{A}_{jk}=0,\quad \forall k \in \{1, \ldots, n \}. 
\end{align}

With the C.G.O-solutions given by \eqref{**} and the integral identity \eqref{*}, they deduce at first, in the case when $m>2$, that $A:=A^{(1)}-A^{(2)}$ is an isotropic matrix, which they fall into the regime of case when $m=2$. Second, by changing the choice of the amplitudes $\ama_1$ and $\ama_1^*$, they prove that $B^{(1)}=B^{(2)}$. Hence, the conclusion that $A^{(1)}=A^{(2)}$ and $q^{(1)}=q^{(2)}$ follows easily.\medskip \\
However, in order to obtain a stability estimate for this inverse boundary value problem, the integral identity \eqref{*} transformed to the following integral inequality 
\begin{align*}
    \abs{\int_\Omega (A D \cdot D u+B \cdot Du+q u)\overline{u^*}dx} \leq \|\Lambda_{A^{(1)}, B^{(1)}, q^{(1)}}-\Lambda_{A^{(2)}, B^{(2)}, q^{(2)}}\| \|u\|_{H^{2m}(\Omega)} \|u^*\|_{H^{2m}(\Omega)}.
\end{align*}
The estimate for the second order perturbation is slightly more complicated. In contrast to the uniqueness result of \cite{B2}, the decomposition \eqref{***} may not be helpful to establish a stability estimate for the 2-tensor field $A$. To remedy this difficulty, we decompose $A$, in the different way, as
\begin{align}
 A=A^\prime +\nabla_{\text{sym}}V+\lambda~ id ,    
\end{align}
where the function $\lambda \in H^{k}(\Omega, \mathbb{C})$, the covector field $V \in  H^{k+1}(\Omega, \mathbb{C}^{n})$ satisfying the boundary condition \eqref{***V} and the tensor field $A^\prime \in  H^k(\Omega, \mathbb{C}^{n^2})$ satisfies the conditions \eqref{***F} and
\begin{align}\label{***A}
 \operatorname{trace}(A^{\prime})=0.
\end{align}
In fact, using the condition \eqref{***A} we are able to estimate $A^\prime:=A^{\prime (1)}-A^{\prime (2)}$. This technique will be discussed later. After this is established, the stability estimates for the vector fields $V$ and $B$ follow by using the Hodge decomposition. Then by changing the choices of amplitudes , we may estimate the function $\lambda$ and therefore, we conclude the stability estimate for $A$. Finally we will combine the estimates we get for $A$ and $B$ to obtain the stability estimate for the potential $q$. For the case where $m=2$, the stability result would follow similarly except in this case $A=\lambda ~id$.
\medskip

The remainder of this paper is organized as follows. In section $2$ we will build complex geometric optics solutions to the perturbed polyharmonic operator. In section $3$, we derive an integral identity involving the perturbations and we give a specific Hodge decomposition of a symmetric tensor $A$ and a vector field $B$. Sections $4$ and $5$, are concerned to prove the stability estimates respectively in case $m>2$ and case $m=2$. In appendix \ref{appendixA}, we develop the proof of Hodge decomposition of a symmetric tensor.
\section{Carleman estimate and C.G.O- solutions}
The main strategy of the proof of stability estimates on determining the symmetric tensor $A$, the vector field $B$ and the electric potential $q$ from the D-to-N map is the use of complex geometrical solutions to estimate the Fourier transform of the difference of two symmetric tensors $A^{(1)}-A^{(2)}$,  the difference of two vector fields $B^{(1)}-B^{(2)}$ and the difference of two potentials $q^{(1)}-q^{(2)}$. The constructed of C.G.O- solutions is based on the use of Carleman estimates. We therefore first outline some known results about Carleman estimate.

\subsection{Local Carleman estimate}
In this section we will first recall the Carleman estimates for semiclassical Lapace operator $(-h^2\Delta)$ and then we use this Carleman estimate to solve an equation involving a conjugated operator of $\LL_{A,B,q}(x,D)$.
\smallskip 

We start by collecting several known Lemmas and notations. Let $\psi\in C^\infty(\R^n,\R)$, consider the conjugated operator with respect the weight function $\psi$
\begin{equation}\label{2.1}
P_\psi(x,hD)=e^{\frac{\psi}{h}}(-h^2\Delta)e^{-\frac{\psi}{h}},\quad h>0.
\end{equation}
Following \cite{KSU}, we say that $\psi$ is a limiting weight function for Laplace operator, if $\nabla \psi \neq 0$ in $\Omega$, and if it satisfies the following Poisson bracket condition
\begin{align*}
    \left\{\overline{p}_\psi, p_\psi \right\}(x, \xi)=0\quad \textrm{where}~~p_\psi(x, \xi)=0,\quad x \in \overline{\Omega}, ~\xi \in \mathbb{R}^{n},
\end{align*}
where $p_\psi$ is the semiclassical principal symbol of the opearator $P_\psi(x, hD)$.\\
Let $\omega\in\s^{n-1}$, the limiting weight function $\psi$ can be chosen as
\begin{equation}\label{2.2}
\psi(x)= x\cdot\omega,\quad x\in\R^n.
\end{equation} 
In what follows we shall equip the standard Sobolev space $H^s(\R^n)$, $s\in\R$, with the semicalssical norm
\begin{equation}\label{2.3}
\norm{u}_{\Hs^{s}(\R^n)}=\norm{\langle hD\rangle ^su}_{L^2(\R^n)},
\end{equation}
here $\langle \xi\rangle=(1+\abs{\xi}^2)^{1/2}$. Let $\varepsilon >0$, we consider the convexified Carleman weight given by 
\begin{align}\label{2.4}
    \psi_\varepsilon=\psi+\frac{h}{2\varepsilon}\psi^2.
\end{align}
Therefore, for $0< h <  \varepsilon \ll 1$, we have 
\begin{align}\label{2.5}
    \| \psi_\varepsilon\|_{W^{2, \infty}(\Omega)}= O(1).
\end{align}

We begin with the local Carleman estimate for semiclassical Laplace operator $(-h^2 \Delta)$ which is due to \cite{ST}.\medskip\\
In what follows, $\les $ holds for $\leq$ modulo a multiplicative positive constant independent of $h$.

\begin{lemma}\label{P2.1}
Let $\psi_\varepsilon$ be the convexified Carleman weight function given by (\ref{2.4}). Then for $0 <h \ll \varepsilon$ and $s \in \mathbb{R}$, we have
\begin{equation}\label{2.6}
\frac{h}{\sqrt{\varepsilon}}\norm{v}_{\Hs^{s+2}(\R^n)}\les \norm{P_{\psi_\varepsilon}(x,hD)v}_{\Hs^s(\R^n)},
\end{equation}
for any $v\in C^\infty_0(\Omega)$.
\end{lemma}
Iterating the Carleman estimate (\ref{2.6}) $m$-times, $m\geq2$, we get the following local Carleman estimate for the polyharmonic operator $(-h^2\Delta)^m$, that is
 \begin{equation}\label{2.7}
(\frac{h}{\sqrt{\varepsilon}})^m\norm{u}_{\Hs^{s+2m}(\R^n)}\les \norm{e^{\frac{\psi_\varepsilon}{h}}(-h^2\Delta)^me^{-\frac{\psi_\varepsilon}{h}}u}_{\Hs^s(\R^n)},
\end{equation}
for any $u\in C^\infty_0(\Omega)$ and $h, \varepsilon>0$ small enough.
\medskip

Armed with inequality (\ref{2.7}) we shall derive a local Carleman estimate for the perturbed polyharmonic operator $\LL_{A,B,q}(x,D)$. Consider the conjugate operator corresponding to the perturbed polyharmonic operator  
\begin{align}\label{2.8}
\hspace{-0.7 cm}\LL_\psi(x,hD)&=e^{\frac{\psi}{h}}(h^{2m}\LL_{A,B,q}(x,D))e^{-\frac{\psi}{h}}\cr
&=e^{\frac{\psi}{h}}(-h^2\Delta)^{m}e^{-\frac{\psi}{h}}+h^{2m-2}e^{\frac{\psi}{h}}A (hD)\cdot (hD) e^{-\frac{\psi}{h}}+h^{2m-1}e^{\frac{\psi}{h}}B\cdot( hD)e^{-\frac{\psi}{h}}+h^{2m}q.
\end{align}

In order to estimate the lower order terms in \eqref{2.8}, we use the following result.

\begin{lemma}\label{L2.2}
Let $t_1$, $t_2 \in \mathbb{R}$ and $V\in H^{t_2}(\Omega)$ with $t_2 > \frac{n}{2}+|t_1|$. Then the estimate 
\begin{align}\label{2.9}
    \| V u \|_{\Hs^{t_1}(\Omega)} \les  \| V \|_{H^{t_2}(\Omega)} \|u\|_{\Hs^{t_1}(\Omega)},
\end{align}
holds  for all $u\in C_0^{\infty}(\Omega)$.
\end{lemma}
\begin{proof}
At first, we extend $V$ to $\R^{n}$ and we consider $\tilde{V}\in H^{t_2}(\R^{n})$ with $\tilde{V}=V$, in $\Omega$, and 
$$\Vert \tilde{V}\Vert_{H^{t_2}(\R^{n})}\les \,\Vert V\Vert_{H^{t_2}(\Omega)}.$$
Then, we get for $u\in C_0^{\infty}(\Omega)$ 
$$\mathcal{F}(\tilde{V}u)(\xi)=(\widehat{\tilde{V}}\ast\widehat{u})(\xi)=\int_{\R^{n}}\widehat{\tilde{V}}(\xi-\eta )\widehat{u}(\eta )\,d\eta, $$
which leads directly to
\begin{equation}\label{A.6*}
\langle h \xi\rangle^{t_1}\mathcal{F}(\tilde{V}u)(\xi)=\int_{\R^{n}}\langle h\xi\rangle^{t_1}\langle h \eta \rangle^{-t_1}\,\widehat{\tilde{V}}(\xi-\eta )\langle h \eta \rangle^{t_1}\widehat{u}(\eta )\,d\eta.
\end{equation}
 Now, we apply the Peetre's inequality (see \cite{CP} page 90) 
\begin{equation}\label{A.7*}
\langle h \xi\rangle^{t_1}\langle h \eta  \rangle^{-t_1}\leq 2^{\abs{t_1}/2}\,\langle \xi-\eta \rangle^{\abs{t_1}/2},\quad \forall\,\xi,\, \eta \in \R^{n}.
\end{equation}
Further, we define two functions $w_1$ and $w_2$ by
$$ w_1(\xi)=\langle h\xi\rangle^{t_1}\vert\widehat{u}(\xi)\vert,\;\;w_2(\xi)=\seq{\xi}^{\abs{t_1}}\,\vert\widehat{\tilde{V}}(\xi)\vert,$$
so that we get using (\ref{A.6*}) and (\ref{A.7*}) that
$$\displaystyle\vert \langle h\xi \rangle^{t_1}\mathcal{F}(\tilde{V}u)(\xi)\vert\les \int_{\R^{n}}w_2(\xi-\eta )\,w_1(\eta )\,d\eta=(w_1 \ast w_2)(\xi).$$
Consequently
\begin{align*}
\displaystyle\Vert \tilde{V}u\Vert_{\Hs^{t_1}(\R^n)}&\les \displaystyle\Vert w_1 \ast w_2\Vert_{L^2(\R^{n})}\\
& \les \,\Vert w_1 \Vert_{L^2(\R^{n})}\,\Vert w_2\Vert_{L^1(\R^{n})}\\
&\les \,\Vert u\Vert_{\Hs^{t_1}(\R^n)}\,\Vert V\Vert_{H^{t_2}(\Omega)}\Big(\int_{\R^{n}}\frac{d\xi}{(1+\vert \xi\vert^2)^{t_2-|t_1|}}\Big)^{1/2}.
 \end{align*}
Since $t_2 > \frac{n}{2}+|t_1|$ then the last integral converge. This ends the proof.
\end{proof}

Armed with Lemma \ref{L2.2}, we are now in position to derive a Carleman estimate to the perturbed polyharmonic operator, which can be stated as follows.

\begin{lemma}\label{P2.3}
Let $A\in \mathcal{A}_{\sigma_1}(M)$, $B \in \mathcal{B}_{\sigma_2}(M)$ and $q \in \mathcal{Q}_{}(M)$ and let $\psi$ be the limiting Carleman weight function given by (\ref{2.2}) and $-2m\leq s\leq 0$. Then the local Carleman estimate
\begin{equation}\label{2.12}
h^m\norm{u}_{\Hs^{s+2m}(\R^n)}\les \norm{\LL_\psi(x,hD)u}_{\Hs^s(\R^n)}\quad m\geq2,
\end{equation}
holds for all $u\in C_0^\infty(\Omega)$ and all $h>0$ small enough.
\end{lemma}

\begin{proof}
To get the Carleman estimate for $\LL_\psi(x,hD)$, we add the lower order perturbation to \eqref{2.7}. First we add the zero order perturbation $q$ then for $-2m \leq s \leq 0$, we get
\begin{align}\label{q1}
    \|h^{2m} q u\|_{\Hs^s(\R^n)} &\leq h^{2m} \|q\|_{L^\infty(\R^n)} \|u\|_{L^2(\R^n)}
    \les h^{2m} \|u\|_{\Hs^{s+2m}(\R^n)}.
\end{align}
Next, we consider the first order perturbation $B$, we compute 
\begin{align*}
    h^{2m-1}e^{\frac{\psi_\varepsilon}{h}}(B\cdot hD)e^{-\frac{\psi_\varepsilon}{h}}u=h^{2m-1}[i(B\cdot \nabla \psi_\varepsilon)u+ B \cdot (hD)u].
\end{align*}
Moreover, we have $B\cdot(h D)u=h D\cdot (Bu)-h(D\cdot B)u$. Then, by using the fact that $\psi_\varepsilon$ satisfies \eqref{2.5}, we obtain, for $0<h<\varepsilon<1$, that
\begin{align}
 \|h^{2m-1}e^{\frac{\psi_\varepsilon}{h}}(B\cdot hD)e^{-\frac{\psi_\varepsilon}{h}}u\|_{\Hs^s(\R^n)}  
  & \les  h^{2m-1} \bigr( \|u\|_{\Hs^{s+2m}(\R^n)}+ \|Bu\|_{\Hs^{s+1}(\R^n)}\bigr).\label{Bu}
\end{align}
Now we estimate the last term in \eqref{Bu}, to do this we distinguish two cases: if $-2m \leq s \leq -1$, we have immediately 
\begin{align}\label{bu_1}
   \|Bu\|_{\Hs^{s+1}(\R^n)} \leq  \|B\|_{L^\infty(\Omega)} \|u\|_{\Hs^{s+2m}(\R^n)}, 
\end{align}
and if $-1 \leq s \leq 0$, we get from Lemma \ref{L2.2}, for $t_2:=\sigma_2 > \frac{n}{2}+1$ and $t_1:=s+1$,
\begin{align}\label{bu_2}
   \|Bu\|_{\Hs^{s+1}(\R^n)} \leq  \|B\|_{H^{\sigma_2}(\Omega)} \|u\|_{\Hs^{s+1}(\R^n)}.
\end{align}
Now plugging \eqref{bu_1} and \eqref{bu_2} into \eqref{Bu}, we conclude for $-2m \leq s \leq 0$,
\begin{align}\label{B_}
     \|h^{2m-1}e^{\frac{\psi_\varepsilon}{h}}(B\cdot hD)e^{-\frac{\psi_\varepsilon}{h}}u\|_{\Hs^s(\R^n)}  \les  h^{2m-1} \|u\|_{\Hs^{s+2m}(\R^n)}.
\end{align}
We end up with the term involving the second order perturbation $A$, by a simple computation we have
\begin{align*}
    h^{2m-2}e^{\frac{\psi_\varepsilon}{h}}A(hD)\cdot (hD) e^{-\frac{\psi_\varepsilon}{h}}u&=h^{2m-2} A  D \psi_\varepsilon \cdot D \psi_\varepsilon u-2h^{2m-2} A  D \psi_\varepsilon \cdot (hD)  u\cr
    &\quad-h^{2m-1} A  D \cdot D \psi_\varepsilon  u+h^{2m-2} A (hD) \cdot (hD) u .
\end{align*}
Since 
$$   A  D \psi_\varepsilon \cdot (hD)u=\sum_{j,k=1}^{n}h D_k(A_{jk}(D_j\psi_\varepsilon) u)-\sum_{j,k=1}^{n}h D_k(A_{jk}D_j\psi_\varepsilon )u,
$$
and
\begin{align*}
A (hD) \cdot (hD) u&=\sum_{j,k=1}^{n} h^2 D _{j}D_{k}(A_{jk}u)-h^2 \sum_{j,k=1}^{n} ( D_{j}D_{k}A_{jk})u-2h\sum_{j,k=1}^{n} 
(D_{j}A_{jk})(h D_{k}u)    \\
&=\sum_{j,k=1}^{n} h^2 D _{j}D_{k}(A_{jk}u)+h^2 \sum_{j,k=1}^{n} ( D_{j}D_{k}A_{jk})u-2h\sum_{j,k=1}^{n} 
h D_{k}\bigr(D_{j}(A_{jk})u\bigr).
\end{align*}
Then we obtain, for $0<h<\varepsilon<1$
\begin{align*}
  & \| h^{2m-2}e^{\frac{\psi_\varepsilon}{h}}A(hD)\cdot (hD) e^{-\frac{\psi_\varepsilon}{h}}u\|_{\Hs^s(\R^n)} \\&\les  h^ {2m-2} \bigr(    \| u\|_{\Hs^{s+2m}}+\sum_{j,k=1}^{n} \bigr(\|A_{jk}(D_j\psi_\varepsilon) u\|_{\Hs^{s+2}}+ \|A_{jk} u\|_{\Hs^{s+2}}
  +h \|D_{j}(A_{jk})  u\|_{\Hs^{s+2}}\bigr)\bigr).
\end{align*}
Again to estimate the last three terms in the above inequality, we distinguish two cases: if $-2 \leq s \leq 0$ , we obtain from Lemma \ref{L2.2} 
\begin{align*}
&\| h^{2m-2}e^{\frac{\psi_\varepsilon}{h}}A(hD)\cdot (hD) e^{-\frac{\psi_\varepsilon}{h}}u\|_{\Hs^s(\R^n)}\cr  &\les h^ {2m-2} \Big(     \| u\|_{\Hs^{s+2m}(\R^n)}+\sum_{j,k=1}^{n}\bigr( \|A_{jk}D_j\psi_\varepsilon  \|_{H^{\sigma_1}(\Omega)}\| u\|_{\Hs^{s+2}(\R^n)}+\|A_{jk} \|_{H^{\sigma_1}(\Omega)}\| u\|_{\Hs^{s+2}(\R^n)}\\&\qquad+\|D_{j}A_{jk} \|_{H^{\sigma_1-1}(\Omega)} \| u\|_{\Hs^{s+2}(\R^n)}\bigr)\Big)
  \\&\les  h^ {2m-2}    \| u\|_{\Hs^{s+2m}(\R^n)}, 
\end{align*}
with $\sigma_1 > \frac{n}{2}+3$. Now if $-2m \leq s < -2$, we get immediately
\begin{align*}
\| h^{2m-2}e^{\frac{\psi_\varepsilon}{h}}A(hD)\cdot (hD) e^{-\frac{\psi_\varepsilon}{h}}u\|_{\Hs^s(\R^n)} &\les 
 h^ {2m-2}    \| u\|_{\Hs^{s+2m}(\R^n)}.
\end{align*}
Then, we conclude for $-2m \leq s \leq 0$,
\begin{align}\label{A_}
    \| h^{2m-2}e^{\frac{\psi_\varepsilon}{h}}A(hD)\cdot (hD) e^{-\frac{\psi_\varepsilon}{h}}u\|_{\Hs^s(\R^n)} &\les h^ {2m-2}    \| u\|_{\Hs^{s+2m}(\R^n)}.
\end{align}
Combining \eqref{q1}, \eqref{B_}  and \eqref{A_} we get from \eqref{2.7} 
\begin{align*}
   ( \frac{h}{\sqrt{\varepsilon}})^m \| u\|_{\Hs^{s+2m}(\R^n)} \les \| h^{2m}e^{\frac{\psi_\varepsilon}{h}}\LL_{A, B, q} e^{-\frac{\psi_\varepsilon}{h}}u\|_{\Hs^{s}(\R^n)}+h^{2m-2}\|u\|_{\Hs^{s+2m}(\R^n)}.
\end{align*}
Choosing now $h < \varepsilon <1$ small enough, we get our desired estimate.
\end{proof}

Let $\LL^*_\psi(x,hD)$ the formal adjoint of $\LL_\psi(x,hD)$ given by $(\LL_\psi u, v   )_{L^2(\Omega)}= ( u, \LL^*_\psi v )_{L^2(\Omega)},~~u,v \in C_{0}^{\infty }(\Omega )$. Notice that $\LL^*_\psi(x,hD)$  has the same form as $\LL_\psi(x,hD)$ except $\psi$ is replaced by $-\psi$ and $A,B$ and $q$ are replaced by
\begin{equation}\label{2.19}  
A_{jk}^*=\overline{A_{jk}},\quad B_k^*=\overline{B_k}+2\sum_{j=1}^n D_j \overline{A_{jk}},\quad q^*=\overline{q}-i\dive\overline{B}+\sum_{j,k=1}^n D_k D_j \overline{A_{jk}},\quad j,k=1, \ldots, n.
\end{equation}
Since $A^*$, $B^*$ and $q^*$ lies in the same admissible sets as $A, B$ and $q$, then we have the same local Carleman estimate for $\LL^*_\psi(x,hD)$.

\begin{lemma}\label{P2.4}
Let $A\in \mathcal{A}_{\sigma_1}(M)$, $B \in \mathcal{B}_{\sigma_2}(M)$ and $q \in \mathcal{Q}_{}(M)$ and let $\psi$ be the limiting Carleman weight function given by (\ref{2.2}) and $-2m\leq s\leq 0$. Then the local Carleman estimate
\begin{equation}\label{2.20}
h^m\norm{u}_{\Hs^{s+2m}(\R^n)}\les \norm{\LL^*_\psi(x,hD)u}_{\Hs^s(\R^n)},
\end{equation}
holds for all $u\in C_0^\infty(\Omega)$, and all $h>0$ small enough.
\end{lemma}
As a consequence of the Carleman estimate (\ref{2.20}) and the Hahn-Banach theorem, we have the following solvability result for $\LL_\psi(x,hD)$. The proof is essentielly well-known and we omit it (see \cite{Gk, K2, K1}).
\begin{lemma}\label{P2.5}
Let $A\in \mathcal{A}_{\sigma_1}(M)$, $B \in \mathcal{B}_{\sigma_2}(M)$ and $q \in \mathcal{Q}_{}(M)$ and let $\psi$ be the limiting Carleman weight function given by (\ref{2.2}). Then for any $\g\in L^2(\Omega)$ the equation
\begin{equation}\label{2.21}
\LL_\psi(x,hD)r=\g\quad \textrm{in}\,\,\Omega,
\end{equation}
has a solution $r\in H^{2m}(\Omega)$, which satisfying the estimate
\begin{equation}\label{2.22}
h^m\norm{r}_{\Hs^{2m}(\Omega)}\les \norm{\g}_{L^2(\Omega)},
\end{equation}
for all $h>0$ small enough.
\end{lemma}

\subsection{Construction of C.G.O- solutions}
We now give the construction of special solution to the polyharmonic equation $\LL_{A,B,q}(x,D)u=0$ in $\Omega$. This proceeds as in \cite{B2} except we precise the constant on the construction of the remainder term in the C.G.O- solutions. The constructed solutions will have the special form 
\begin{equation}\label{2.23}
u(x,h)=e^{\frac{\varphi}{h}}(\ama_1(x)+h \ama_2(x)+r(x,h)),\quad x\in\Omega,
\end{equation}
where $\varphi(x)=\psi(x)+i\widetilde{\psi}(x)$, $\psi$ and $\widetilde{\psi}$ are the real valued functions given by
\begin{equation}\label{2.24}
\psi(x)=x\cdot\omega,\quad \widetilde{\psi}(x)=x\cdot\tilde{\omega},
\end{equation}
and $\omega,\widetilde{\omega}\in\s^{n-1}$, with $\omega\cdot\tilde{\omega}=0$. Let $\varrho=\omega+i\tilde{\omega}$, the functions $\ama_1$ and $\ama_2$ as an amplitude which respectively solve the transport equations

\begin{equation}\label{2.25}
T_{\varrho}(x,D)\ama_1:=(\varrho\cdot\nabla)^m\ama_1=0,\quad \textrm{in}\,\,\Omega,\quad\textrm{for}~m \geq 2,
\end{equation}
and 
\begin{equation}\label{2.26}
(2\varrho\cdot\nabla)^m\ama_2=\left\{\begin{matrix}
\bigr(-4\Delta(\varrho \cdot \nabla)+2A\varrho \cdot \nabla +i \varrho \cdot B\bigr) \ama_1,&\quad \textrm{if}~~m=2, \\
\bigr(-12\Delta(\varrho \cdot \nabla)^{2}-A\varrho \cdot \varrho\bigr)\ama_1,&\quad \textrm{if}~~m=3,\\
-m\Delta(2\varrho \cdot \nabla)^{m-1}\ama_1,&\quad \textrm{if}~~ m>3.
\end{matrix}\right.
\end{equation}
Finally, the correction term $r(\cdot, h)\in H^{2m}(\Omega)$.\\
The solution of the transport equation \eqref{2.25} always exists and can be taken in $C^\infty(\overline{\Omega})$. We will present the existence of complex amplitude $\ama_2$ satisfying \eqref{2.26} in Subsection \ref{Sect3}.

\begin{lemma}\label{L2.6}
Let $A\in \mathcal{A}_{\sigma_1}(M)$, $B \in \mathcal{B}_{\sigma_2}(M)$ and $q \in \mathcal{Q}_{}(M)$ and $\varphi$ as above. Then for any $\ama_1 \in C^\infty(\overline{\Omega})$ and $\ama_2\in H^{2m}(\Omega)$ satisfying respectively the transport equations (\ref{2.25}) and \eqref{2.26} and for all $h>0$ small enough, we can construct a solution to $\LL_{A,B,q}(x,D)u=0$, $x\in\Omega$, of the form
\begin{equation}\label{2.27}
u(x,h)=e^{\frac{\varphi}{h}}(\ama_1(x)+h \ama_2(x)+r(x,h)),\quad x\in\Omega,
\end{equation}
where $\varphi(x)=\psi(x)+i\tild{\psi}(x)=x\cdot\varrho$. Moreover the correction term $r(\cdot,h)\in H^{2m}(\Omega)$, satisfying 
\begin{equation}\label{2.28}
\norm{r(\cdot,h)}_{\Hs^{2m}(\Omega)}\les h^2 (\|\ama_1\|_{H^{2m}(\Omega)}+\| \ama_2\|_{H^{2m}(\Omega)}).
\end{equation}
\end{lemma}

\begin{proof}
Let $\ama_1\in C^\infty(\overline{\Omega})$ and $\ama_2\in H^{2m}(\Omega)$ satisfying respectively the transport equations (\ref{2.25}) and \eqref{2.26}. We denote $\g_h\in L^2(\Omega)$ given by
\begin{equation}\label{2.29}
\g_h(x)=-e^{i\frac{\tild{\psi}}{h}}\LL_{-\varphi}(x,hD)\bigr(\ama_1(x)+h \ama_2(x)\bigr),\quad x\in\Omega,
\end{equation}
where $\LL_{-\varphi}(x,hD)$ denote the conjugated operator
\begin{equation}\label{2.30}
\LL_{-\varphi}(x,hD)=e^{-\frac{\varphi}{h}}h^{2m}\LL_{A,B,q}(x,D)e^{\frac{\varphi}{h}}.
\end{equation}
By a simple computation, we get
\begin{align*}
\hspace{-2 cm}-e^{-i\frac{\tild{\psi}}{h}}\g_h&= \Big((-h^2\Delta-2h\varrho\cdot\nabla)^m+h^{2m-2}A(hD-i\varrho)\cdot (hD-i\varrho)+h^{2m-1}B\cdot(hD-i\varrho)\cr
&\qquad +h^{2m}q\Big)(\ama_1+h \ama_2)\cr
&=\Big((-1)^m \sum_{k=0}^{m}\binom{m}{k}h^{2m-k}(\Delta)^{m-k}(2\varrho\cdot\nabla)^{k} + h^{2m} ( AD\cdot D+B\cdot D+q )\cr
&\quad +h^{2m-1} (-2i A  \varrho \cdot D-i \varrho \cdot B )-h^{2m-2} (A \varrho \cdot \varrho ) \Big)(\ama_1+h \ama_2)\cr
&= h^{2m}\bigr( \LL_{A, B, q}(x, D) (\ama_1+h \ama_2)-2i A \varrho \cdot D\ama_2-i \varrho \cdot B\ama_2\bigr)\cr
&\quad+(-1)^m \sum_{k=1}^{m}\binom{m}{k}h^{2m-k}(\Delta)^{m-k}(2\varrho\cdot\nabla)^{k}(\ama_1+h \ama_2)\cr&\quad+h^{2m-1} \bigr(-2i A  \varrho \cdot D\ama_1-i \varrho \cdot B \ama_1-A \varrho \cdot \varrho \ama_2\bigr)-h^{2m-2} A \varrho \cdot \varrho \ama_1.
\end{align*}
Since $\ama_1$ satisfying the transport equation \eqref{2.25}, we get
\begin{align}\label{2.31}
-e^{-i\frac{\tild{\psi}}{h}}\g_h&=  h^{2m}\bigr( \LL_{A, B, q}(x, D) (\ama_1+h \ama_2)-2i A  \varrho \cdot D\ama_2-i \varrho \cdot B\ama_2\bigr)\cr
&\quad+(-1)^m \sum_{k=1}^{m-1}\binom{m}{k}h^{2m-k}(\Delta)^{m-k}(2\varrho\cdot\nabla)^{k}(\ama_1+h \ama_2)+ h^{m+1}(-1)^m  (2\varrho\cdot\nabla)^{m}\ama_2\cr
&\quad+h^{2m-1} \bigr(-2i A \varrho \cdot D\ama_1-i \varrho \cdot B \ama_1-A \varrho \cdot \varrho \ama_2\bigr)-h^{2m-2} A \varrho \cdot \varrho \ama_1.
\end{align}
Now we need to show that the terms in right hand side of \eqref{2.31} are terms of order $\mathcal{O}(h^{m+2})$, $m \geq 2$, to do this we distinguish two cases. First, if $m=2$, in that case using the fact that $A$ is isotropic we find out from \eqref{2.31} that
\begin{align*}
-e^{-i\frac{\tild{\psi}}{h}}\g_h&=  h^{4}\bigr( \LL_{A, B, q}(x, D) (\ama_1+h \ama_2)-2i A  \varrho \cdot D\ama_2-i \varrho \cdot B\ama_2+4\Delta (\varrho\cdot\nabla) \ama_2\bigr)\cr&\quad+ h^{3}\bigr( (2\varrho\cdot\nabla)^{2}\ama_2-2i A  \varrho \cdot D\ama_1-i \varrho \cdot B \ama_1+4\Delta (\varrho\cdot\nabla) \ama_1\bigr).
\end{align*}
In order to get $\|e^{-i\frac{\tild{\psi}}{h}}\g_h\|_{L^2(\Omega)}=\mathcal{O}(h^{4})$, we choose the amplitude $\ama_2 \in H^4(\Omega)$ such that 
\begin{align*}
   (2\varrho\cdot\nabla)^{2}\ama_2=2i A  \varrho \cdot D\ama_1+i \varrho \cdot B \ama_1-4\Delta (\varrho\cdot\nabla) \ama_1.
\end{align*}
Having chosen $\ama_2$ in this way, we obtain for $m=2$, the
following equation
\begin{align}\label{2.32}
-e^{-i\frac{\tild{\psi}}{h}}\g_h=  h^{2m}\bigr( \LL_{A, B, q}(x, D) (\ama_1+h \ama_2)-2i A  \varrho \cdot D\ama_2-i B\cdot \varrho\ama_2+4\Delta (\varrho\cdot\nabla) \ama_2\bigr).
\end{align}
Now, if $m>2$ we get from \eqref{2.31}
\begin{align}\label{2.33}
-e^{-i\frac{\tild{\psi}}{h}}\g_h&=  h^{2m}\bigr( \LL_{A, B, q}(x, D) (\ama_1+h \ama_2)-2i A  \varrho \cdot D\ama_2-i \varrho \cdot B\ama_2\bigr)\cr
&\quad+(-1)^m \sum_{k=1}^{m-2}\binom{m}{k}h^{2m-k}(\Delta)^{m-k}(2\varrho\cdot\nabla)^{k}(\ama_1+h \ama_2)\cr&\quad+
h^{m+2}(m(-1)^m \Delta (2\varrho\cdot\nabla)^{m-1}\ama_2)\cr
&\quad+h^{m+1}\bigr((-1)^m  (2\varrho\cdot\nabla)^{m}\ama_2+(-1)^m m \Delta  (2\varrho\cdot\nabla)^{m-1}\ama_1\bigr)\cr&\quad+h^{2m-1} \bigr(-2i A \varrho \cdot D\ama_1-i \varrho \cdot B \ama_1-A \varrho \cdot \varrho \ama_2\bigr)-h^{2m-2} A \varrho \cdot \varrho \ama_1.
\end{align}
We observe that for $1 \leq k \leq m-2$, we have $h^{2m-k} \leq h^{m+2}$, then we deduce
\begin{align*}
 -e^{-i\frac{\tild{\psi}}{h}}\g_h&=  \mathcal{O}(h^{m+2}) +h^{m+1}\bigr((-1)^m  (2\varrho\cdot\nabla)^{m}\ama_2+(-1)^m m \Delta  (2\varrho\cdot\nabla)^{m-1}\ama_1\bigr)-h^{2m-2} A \varrho \cdot \varrho \ama_1.   
\end{align*}
Notice that $h^{2m-2}=h^{m+1}$ for $m=3$, and $h^{2m-2} \leq h^{m+2}$ for $m>3$. Then in order to obtain $ -e^{-i\frac{\tild{\psi}}{h}}\g_h=  \mathcal{O}(h^{m+2})$, we distinguish again two cases : first, if $m=3$, we get from \eqref{2.33}
\begin{align*}
 -e^{-i\frac{\tild{\psi}}{h}}\g_h&= h^{6}\bigr( \LL_{A, B, q}(x, D) (\ama_1+h \ama_2)-2i A  \varrho \cdot D\ama_2-i \varrho \cdot B\ama_2\bigr)\cr&\quad+h^{5}\bigr(- 3 (\Delta)^{2}(2\varrho\cdot\nabla)(\ama_1+h \ama_2)
-3 \Delta (2\varrho\cdot\nabla)^{2}\ama_2-2i A  \varrho \cdot D\ama_1-i \varrho \cdot B \ama_1-A \varrho \cdot \varrho \ama_2\bigr)\cr&\quad+h^{4}\bigr(- (2\varrho\cdot\nabla)^{3}\ama_2-3  \Delta  (2\varrho\cdot\nabla)^{2}\ama_1- A \varrho \cdot \varrho \ama_1\bigr).  
\end{align*}
We choose the amplitude $\ama_2 \in H^6(\Omega)$ satisfies
\begin{align*}
   (2\varrho\cdot\nabla)^{3}\ama_2=-3  \Delta  (2\varrho\cdot\nabla)^{2}\ama_1- A \varrho \cdot \varrho \ama_1,\quad \textrm{in}~\Omega. 
\end{align*}
Hence, we can immediately conclude that
\begin{align}\label{2.34}
  -e^{-i\frac{\tild{\psi}}{h}}\g_h&= h^{6}\bigr( \LL_{A, B, q}(x, D) (\ama_1+h \ama_2)-2i A  \varrho \cdot D\ama_2-i \varrho \cdot B\ama_2\bigr)\cr&\quad+h^{5}\bigr(- 3 (\Delta)^{2}(2\varrho\cdot\nabla)(\ama_1+h \ama_2)
-3 \Delta (2\varrho\cdot\nabla)^{2}\ama_2\cr &\quad-2i A  \varrho \cdot D\ama_1-i \varrho \cdot B \ama_1-A\varrho \cdot \varrho \ama_2\bigr). 
\end{align}
Finally, if $m>3$, we choose the amplitude $\ama_2 \in H^{2m}(\Omega)$ satisfies
\begin{align*}
 (2\varrho\cdot\nabla)^{m}\ama_2=- m \Delta  (2\varrho\cdot\nabla)^{m-1}\ama_1,\quad \textrm{in}~\Omega, \end{align*}
which allows us to obtain from \eqref{2.33} 
 \begin{align}\label{2.35}
 -e^{-i\frac{\tild{\psi}}{h}}\g_h&=  h^{2m}\bigr( \LL_{A, B, q}(x, D) (\ama_1+h \ama_2)-2i  \varrho \cdot D\ama_2-i  \cdot \varrho\ama_2\bigr)\cr &\quad+h^{m+2}(m(-1)^m \Delta (2\varrho\cdot\nabla)^{m-1}\ama_2)+(-1)^m \sum_{k=1}^{m-2}\binom{m}{k}h^{2m-k}(\Delta)^{m-k}(2\varrho\cdot\nabla)^{k}(\ama_1+h \ama_2)\cr&\quad+
h^{2m-1} \bigr(-2i A  \varrho \cdot D\ama_1-i \varrho \cdot B \ama_1-A \varrho \cdot \varrho \ama_2\bigr)-h^{2m-2} A \varrho \cdot \varrho \ama_1.
\end{align}
We deduce from \eqref{2.32}, \eqref{2.34} and \eqref{2.35} that
\begin{equation}\label{2.36}
\norm{e^{-i\frac{\widetilde{\psi}}{h}}\g_h}_{L^2(\Omega)}\les h^{m+2}(\|\ama_1\|_{H^{2m}(\Omega)}+\| \ama_2\|_{H^{2m}(\Omega)}),\quad m \geq 2.
\end{equation}
So, by Lemma \ref{P2.5}, for $h>0$ sufficiently small, we can conclude that there exists $\tild{r}(\cdot,h)\in H^{2m}(\Omega)$ solving
$$
\LL_{-\psi}(x,hD)\tild{r}(x,h)=\g_h(x)\quad \textrm{in} \,\,\Omega,
$$ 
such that 
\begin{equation*}
h^m\norm{\tild{r}(\cdot,h)}_{\Hs^{2m}(\Omega)}\les \norm{\g_h}_{L^2(\Omega)}\leq  h^{m+2}(\|\ama_1\|_{H^{2m}(\Omega)}+\| \ama_2\|_{H^{2m}(\Omega)}).
\end{equation*}
We denote $r(\cdot,h)=e^{-i\frac{\tild{\psi}}{h}}\tild{r}(\cdot,h)$. Therefore
\begin{equation*}
\LL_{A,B,q}(x,D)(e^{\frac{\varphi}{h}}(\ama_1(x)+h \ama_2(x)+  r(x,h))=0,
\end{equation*}
with $r(\cdot, h) \in H^{2m}(\Omega)$, such that
\begin{equation}
\norm{r(\cdot,h)}_{\Hs^{2m}(\Omega)}\les  h^2 (\|\ama_1\|_{H^{2m}(\Omega)}+\| \ama_2\|_{H^{2m}(\Omega)}).
\end{equation}
The proof is complete.
\end{proof}

Similarly, let consider the following transport equations
\begin{equation}\label{2.38}
T^*_{\varrho}(x,D)\ama_1^*:=(\overline{\varrho}\cdot\nabla)^m \ama_1^*=0,\quad \textrm{in}\,\,\Omega,
\end{equation}
and 
\begin{equation}\label{2.39}
(2\overline{\varrho}\cdot\nabla)^m\ama_2^*=\left\{\begin{matrix}
\bigr(4\Delta(\overline{\varrho} \cdot \nabla)-2{A^*}\overline{\varrho} \cdot \nabla -i {B^*} \cdot \varrho\bigr) \ama_1^*,&\quad \textrm{if}~~m=2, \\
\bigr(-12\Delta(\overline{\varrho} \cdot \nabla)^{2}-{A^*}\overline{\varrho} \cdot \overline{\varrho}\bigr)\ama_1^*,&\quad \textrm{if}~~m=3,\\
-m\Delta(-2\overline{\varrho} \cdot \nabla)^{m-1}\ama_1^*,&\quad \textrm{if}~~ m>3.
\end{matrix}\right.
\end{equation}
where $\varrho=\omega+i\tilde{\omega}$ and $A^*$, $B^*$ are given by \eqref{2.19}. We can similarly prove the existence of C.G.O solutions to the adjoint operator $\LL^*_{A,B,q}(x,D)$.

\begin{lemma}\label{L2.7}
Let $A\in \mathcal{A}_{\sigma_1}(M)$, $B \in \mathcal{B}_{\sigma_2}(M)$ and $q \in \mathcal{Q}_{}(M)$ and $\varphi$ as above. Then for any $\ama_1^*\in C^\infty(\overline{\Omega})$ and $\ama_2^*\in H^{2m}(\Omega)$ satisfying respectively the transport equations (\ref{2.38}) and \eqref{2.39} and all $h>0$ small enough, we can construct a solution to $\LL^*_{A,B,q}(x,D)u=0$, $x\in\Omega$, of the form
\begin{equation}\label{2.40}
u^*(x,h)=e^{-\frac{\overline{\varphi}}{h}}(\ama_1^*(x)+h\ama_2^*(x)+r^*(x,h)),\quad x\in\Omega,
\end{equation}
where $\varphi(x)=\psi(x)+i\widetilde{\psi}(x)=x\cdot\varrho$. Moreover the correction term $r^*(\cdot,h)\in H^{2m}(\Omega)$, satisfying 
\begin{equation}\label{2.41}
\norm{r^*(\cdot,h)}_{\Hs^{2m}(\Omega)}\les h^2(\|\ama_1^*\|_{H^{2m}(\Omega)}+\| \ama_2^*\|_{H^{2m}(\Omega)}).
\end{equation}
\end{lemma}

\subsection{Solvability of the transport equation}\label{Sect3}

We will now establish the existence of complex amplitude $\ama_2$ satisfying \eqref{2.26}. We will follow the same approach used in \cite{IS}. Let $P(D)$ be a differential operator with constant coefficients \begin{align}
    P(D)=\sum_{\abs{\ama}\leq m} a_\ama D^\ama,\quad D=-i\nabla,\quad a_\ama \in \C.
\end{align}
We associate to the operator $P(D)$ its full symbol $p(\xi)$ given by 
\begin{align}
    p(\xi)=\sum_{\abs{\ama}\leq m} a_\ama \xi^\ama,\quad \xi \in \R^n.
\end{align}
Moreover, we set 
\begin{align}\label{p_tilde}
    \widetilde{p}(\xi)=\bigr(\sum_{\abs{\ama}\leq m} \abs{\partial_\xi^\ama p(\xi)}^2 \bigr)^{\frac{1}{2}},\quad \xi \in \R^n.
\end{align}
Let us first recall the following result due to \cite{IS}, where it is obtained as a consequence of the general theory of \cite{Hor}.

\begin{lemma}\label{L2.8}
Let $P(D) \neq 0$, be a differential operator with constant coefficients. Then for all $k \in \mathbb{N}$, there exists a linear operator $E \in \mathcal{B}(H^k(\Omega))$ such that $P(D)(E(f))=f$, for all $f \in H^k(\Omega)$. Moreover, there exists $C >0$ such that 
\begin{align}\label{2.45}
    \|E(f)\|_{H^k(\Omega)} \leq C \bigr(\underset{\xi \in \R^n}{\operatorname{sup}}  \frac{1}{\widetilde{p}(\xi)} \bigr)  \|f\|_{H^k(\Omega)}, \quad \forall f \in H^k(\Omega).
\end{align}
Here $C$ depends only on the order of $P(D)$, $n$ and $\Omega$.
\end{lemma}

For $\varrho=\omega + i \widetilde{\omega} \in \mathbb{S}^{n-1}+i \mathbb{S}^{n-1}$ as above, we consider the differential operator with constant coefficients
\begin{align}
   T_{\varrho}(D)=(\varrho\cdot\nabla)^m , \quad \textrm{for}~~m \geq 2.
\end{align}
As a consequence of the above Lemma we have the following result.

\begin{lemma}\label{L2.9}
The operator $ T_{\varrho}(D)$ admits a bounded inverse $ E_{\varrho}$ from $H^k(\Omega)$ to $H^k(\Omega)$, such that, if $f \in H^k(\Omega)$ and $v=E_{\varrho}f  \in H^k(\Omega)$, we have 
\begin{align}\label{2.47}
    \|v\|_{H^k(\Omega)} \leq C \|f\|_{H^k(\Omega)}.
\end{align}
Here $C >0$ is a positive constant depends only on $n$, $m$ and $\Omega$.
\end{lemma}

\begin{proof}
From Lemma \ref{L2.8}, there exists a linear operator $E_{\varrho}\in \mathcal{B}(H^k(\Omega))$ such that
$$T_{\varrho}\circ E_{\varrho}f=f, \quad\textrm{for any } f \in H^k(\Omega).$$
Moreover, since the symbol of the differential operator $T_{\varrho}(D)$ is given by $$p_{\varrho}(\xi)=i^m(\varrho \cdot \xi)^m,\quad \xi \in \R^n.$$
So we have, by \eqref{p_tilde} and H\"older's inequality 
$$(\widetilde{p}_{\varrho}(\xi))^2 \geq   \sum_{j=1}^n \abs{m!  \varrho_j^m}^2 \geq \frac{(m!)^2}{n^{m-1}}\bigr(\sum_{j=1}^n \abs{\varrho_j }^2\bigr)^m \geq \frac{2^m (m!)^2}{n^{m-1}}.$$
Then, from \eqref{2.45} we get 
$$ \|E_{\varrho}f\|_{H^k(\Omega)} \leq C \|f\|_{H^k(\Omega)}.$$
This completes the proof of the Lemma.
\end{proof}

We are in the position now to establish the existence of amplitude $\ama_2$ satisfying \eqref{2.26}.

\begin{lemma}\label{L2.10}
Let $A\in \mathcal{A}_{\sigma_1}(M)$, $B \in \mathcal{B}_{\sigma_2}(M)$ and $\ama_1 \in C^\infty(\overline{\Omega})$ satisfies \eqref{2.25}. Then, there exists $\ama_2 \in H^{2m}(\Omega)$ solving the transport equation
\eqref{2.26}.
Moreover, there exists $C>0$ such that we have
\begin{align}\label{2.50}
    \|\ama_2\|_{H^{2m}(\Omega)} \leq C\big( \|\ama_1\|_{H^{2m}(\Omega)}+\|(\varrho \cdot \nabla)^{m-1} \ama_1\|_{H^{2m+2}(\Omega)}\big).
\end{align}
Here $C$ is a positive constant depends only on $n$, $m$ and $\Omega$.
\end{lemma}
\begin{remark}
Notice from \eqref{2.26} and Lemma \eqref{L2.9}, that the regularity of the complex
amplitude $\ama_2$
depends on the regularity of the coefficients $A$ and $B$. This justifies our extra regularity assumption on the coefficients. We hope that the regularity condition imposed on $2-$tensor field can be weakened from $W^{\operatorname{min}(2m, 6), \infty}(\Omega)$ to $W^{2, \infty}(\Omega)$, if we replace $A$ in \eqref{2.26} by its regularization $A_h$ given by 
$$A_h=\chi_h \ast A,$$ 
where $\chi_h(x)=h^{-n\sigma } \chi(h^{-\sigma }x)$ is the usual mollifier and $\sigma \in (0,1/2)$.
\end{remark}

In a similar manner, we can prove that

\begin{lemma}\label{L2.11}
Let $A\in \mathcal{A}_{\sigma_1}(M)$, $B \in \mathcal{B}_{\sigma_2}(M)$ and $\ama_1^* \in C^\infty(\overline{\Omega})$ satisfies \eqref{2.38}. Then, there exists $\ama_2^* \in H^{2m}(\Omega)$ solving the transport equation \eqref{2.39}. Moreover, there exists $C>0$ such that we have
\begin{align}\label{2.53}
    \|\ama_2^*\|_{H^{2m}(\Omega)} \leq C \big(\|\ama_1^*\|_{H^{2m}(\Omega)}+\|(\overline{\varrho} \cdot \nabla)^{m-1} \ama_1^*\|_{H^{2m+2}(\Omega)}\big).
\end{align}
Here $C$ is a positive constant depends only on $n$, $m$ and $\Omega$.
\end{lemma}
For the sake of simplicity, the norm in right hand side of \eqref{2.50} and \eqref{2.53} is denoted by
\begin{align}
    N_{2m, \varrho}(\ama):=\|\ama\|_{H^{2m}(\Omega)}+\|(\varrho \cdot \nabla)^{m-1} \ama\|_{H^{2m+2}(\Omega)}.
\end{align}

\section{Integral estimates and Hodge decomposition}
In this section, we will use the properties of the Dirichlet-to-Neumann map to prove an integral estimate, which relates the difference of coefficients to the difference of the corresponding Dirichlet-to-Neumann map. This integral estimate will be our starting point in proving the stability estimates for the inverse problem under consideration. In the second part of this section we give a specific Hodge decomposition of a symmetric tensor $A$ and a vector field $B$. \smallskip \\
As in the first section we consider a priori constant $M>0$, $\sigma_1 > \frac{n}{2}+3$ and $\sigma_2 > \frac{n}{2}+1$. Also, we consider for $j=1,2$, $A^{(j)} \in \mathcal{A}_{\sigma_1}(M)\cap \mathcal{E}^\prime(\Omega)$, $B^{(j)} \in \mathcal{B}_{\sigma_2}(M)\cap \mathcal{E}^\prime(\Omega)$ and $q^{(j)}\in \mathcal{Q}_{}(M)$ pair of admissible coefficients. We denote
\begin{equation}\label{3.1}
A=A^{(2)}-A^{(1)},\qquad B=B^{(2)}-B^{(1)}\quad\text{and  }~~~~ q=q^{(2)}-q^{(1)}.
\end{equation}
We extend $A$, $B$ and $q$ by zero outside
$\Omega$. In the sequel of the next, we denote by $A$, $B$ and $q$ these extensions and by $\LL_j(x,D)$, for $j=1,2,$ the operator corresponding to the perturbation $(A^{(j)},B^{(j)},q^{(j)})$, that is
\begin{equation}
\LL_{j}(x,D)=\LL_{A^{(j)},B^{(j)},q^{(j)}}(x,D),\quad j=1,2.
\end{equation}
For notational convenience, the Dirichlet to Neumann map is denoted by
\begin{equation}
\Lambda^{(j)}=\Lambda_{A^{(j)},B^{(j)},q^{(j)}},\quad j=1,2.
\end{equation}
The key idea in the stability result is to use complex geometric optics solutions $u$ to $\LL_2(x, D)u=0$ in $\Omega$ and $u^*$ to $\LL_1^*(x, D)u^*=0$ in $\Omega$ and plug them in an integral identity. \smallskip\\
For $A^{(1)}$, $B^{(1)}$ and $q^{(1)}$ as above. In what follows,
we shall need the generalization of Green's formula given in \cite{AS},
\begin{align}\label{3.4}
&\left(\LL_{1}(x,D)w , v \right)-\left (w , \LL^*_{1}(x,D)v \right)\cr&=\sum_{\ell=1}^{m}\int_\Gamma \left [ \bigr((-\Delta )^{m-\ell}w\bigr)\bigr(\overline{\partial_\nu (-\Delta)^{\ell-1}v}\bigr) - \bigr(\partial_\nu(-\Delta )^{m-\ell}w\bigr)\bigr(\overline{(-\Delta)^{\ell-1}v}\bigr) \right ]ds_x\cr
&:= \left<\gamma^t w , \tilde{\gamma}v\right> -\bigr< \tilde{\gamma}w, \gamma^t v \bigr>, 
\end{align}
where  $\gamma^t w=((-\Delta)^{m-1}  w, \ldots,(-\Delta) w, w)$. Here $\left(\cdot , \cdot \right)$ and $\left<\cdot , \cdot
\right>$ denote the $L^2-$ inner product respectively on $\Omega$ 
and $\Gamma$, and $ds_x$ stands for the Euclidean surface measure on $\Gamma$.
\smallskip

We start by the following integral identity.
\begin{lemma}\label{L3.1}
Let $A$, $B$ and $q$ given by \eqref{3.1}. For given $f=(f_0, \ldots, f_{m-1}), f^*=(f_0^*, \ldots, f_{m-1}^*)\in \mathcal{H}^{m,\frac{1}{2}}(\Gamma)$, let $u,\,u^*\in H^{2m}(\Omega)$ solutions respectively, to
\begin{equation}
\left\{\begin{array}{ll}
\LL_{2}(x,D)u=0 & \textrm{in}\,\Omega,\cr
\gamma u=f & \textrm{on}\,\Gamma,
\end{array}
\right.
;\qquad 
\left\{\begin{array}{ll}
\LL^*_{1}(x,D)u^*=0 & \textrm{in}\,\Omega,\cr
\gamma u^*=f^* & \textrm{on}\,\Gamma.
\end{array}
\right.
\end{equation}
Then the following identity holds true
\begin{equation}
\int_\Omega (A(x)D \cdot D u+B(x)\cdot Du+q(x)u)\overline{u^*}dx=\bigr<(\Lambda^{(2)}-\Lambda^{(1)})f, (f^*)^t\bigr>,\label{3.6}
\end{equation}
where $(f^*)^t=(f_{m-1}^*, \ldots, f_0^*)$.
\end{lemma}

\begin{proof}
Let $u$ and $u^*$ as in the lemma. Let $u_1\in H^{2m}(\Omega)$ be the unique solution of
\begin{equation}\label{3.7}
\left\{\begin{array}{ll}
\LL_{1}(x,D)u_1=0 & \textrm{in}\,\Omega,\\
\gamma u_1=f & \textrm{on}\,\Gamma.
\end{array}
\right.
\end{equation}
Putting $w=u_1-u$, then $w$ solves 
\begin{equation}\label{3.8}
\left\{\begin{array}{ll}
\LL_{1}(x,D)w=A(x)D\cdot D u+B(x)\cdot Du+q(x)u & \textrm{in}\,\Omega,\cr
\gamma w=0 & \textrm{on}\,\Gamma.
\end{array}
\right.
\end{equation}
Next, we multiply the first equation in (\ref{3.8}) by $\overline{u^*}$ and we then apply the Green's formula \eqref{3.4} to get
\begin{equation}
\int_\Omega (A(x)D\cdot D u+B(x)\cdot Du+q(x)u)\overline{u^*}dx=\int_\Gamma(\Lambda^{(2)}-\Lambda^{(1)})f\cdot\overline{(f^*)^t}ds_x.
\end{equation}
This completes the proof.
\end{proof}

For given $\omega,\tild{\omega}\in\s^{n-1}$ with $\omega\cdot\tild{\omega}=0$, let $\varrho=\omega+i\tild{\omega}$ and let $\varphi(x)=x\cdot\varrho$. For $h>0$ sufficiently small, Lemmas \ref{L2.6} and \ref{L2.7} guarantee the existence of C.G.O- solutions $u\in H^{2m}(\Omega)$ verifying $\LL_2(x,D)u=0$ in $\Omega$ and $u^*\in H^{2m}(\Omega)$ verifying $\LL_1^*(x,D)u^*=0$ in $\Omega$ and such that
\begin{align}
u(x,h)&=e^{\frac{\varphi}{h}}(\ama_1(x)+h\ama_2(x)+r(x,h)),\label{3.10}\\
u^*(x,h)&=e^{-\frac{\overline{\varphi}}{h}}(\ama_1^*(x)+h \ama_2^*(x)+r^*(x,h)),\label{3.11}
\end{align}
where $r(\cdot,h)$ and $r^*(\cdot,h)$ satisfy
\begin{align}
\norm{r(\cdot,h)}_{\Hs^{2m}(\Omega)}&\les h^2N_{2m,\varrho}(\ama_1),\label{3.12} \\ \norm{r^*(\cdot,h)}_{\Hs^{2m}(\Omega)}&\les h^2N_{2m,\overline{\varrho }}(\ama_1^*),\label{3.12'}
\end{align}
and where $\ama_1(\cdot,\varrho)$ and $\ama_1^*(\cdot,\varrho)$ belong to $C^\infty(\overline{\Omega})$ and satisfying the following transport equations respectively
\begin{align}
& (\varrho\cdot\nabla)^m \ama_1 =0\quad \mathrm{in}\,\,\Omega, \label{3.13}\\
& (\overline{\varrho}\cdot\nabla)^m\ama_1^* =0\quad \mathrm{in}\,\,\Omega.\label{3.14}
\end{align}

 As a consequence of Lemma \ref{L3.1} and the C.G.O- solutions constructed as above, we have the following integral identity.

\begin{lemma}\label{L3.2}
Let $A$, $B$ and $q$ given by \eqref{3.1}. For all $\ama_1,\,\ama_1^*\in C^\infty(\overline{\Omega})$ satisfy \eqref{3.13} and \eqref{3.14} respectively, the following identity holds true
\begin{equation}
\int_\Omega \bigr(h^{-1} A \varrho \cdot \varrho \ama_1+2 A \varrho \cdot \nabla \ama_1 + i (\varrho \cdot B) \ama_1\bigr)\overline{\ama_1^*}dx=\mathcal{I}(h),\label{3.15}
\end{equation}
where
\begin{align}\label{3.16}
 |  \mathcal{I}(h)| 
  &\les (e^{C/h} \|\Lambda^{(1)}-\Lambda^{(2)}\|+h+\|A \varrho \cdot \varrho\|_{L^\infty})N_{2m, \varrho}(\ama_1)N_{2m, \overline{\varrho }}(\ama_1^*).
\end{align}
Here $C$ is a positive constant independent of $h$.
\end{lemma}

\begin{proof}
Let $u$ be a solution to $\LL_2(x, D)u=0 $, in $\Omega$, of the form (\ref{3.10}) and $u^*$ be a solution to $\LL_1^*(x, D)u^*=0 $, in $\Omega$, of the form (\ref{3.11}). Thus, we get
\begin{align}
hA D \cdot D u&=e^{\frac{\varphi}{h}}\big(-h^{-1}(A\varrho\cdot\varrho)-2iA\varrho\cdot D+AD\cdot(hD)\big)(\ama_1+h\ama_2+r).\label{3.17}
\end{align}
Then the first term in the left hand side of (\ref{3.6}) becomes
\begin{equation}
h\int_\Omega (A D \cdot D u)\overline{u^*}dx=-
\int_\Omega \bigr(h^{-1}(A\varrho\cdot\varrho)\ama_1 +2i A\varrho\cdot D\ama_1\bigr)\overline{\ama_1^*}dx +\mathcal{J}_1(h),\label{3.18}
\end{equation}
where $\mathcal{J}_1$ denotes the integral
\begin{align}
\mathcal{J}_1(h)&=-\int_\Omega \bigr(h^{-1}A\varrho\cdot\varrho(h\ama_2+r)+2i A \varrho \cdot D (h\ama_2+r)\bigr) \overline{(\ama_1^*+h\ama_2^*+r^*)}dx\cr&\quad-\int_\Omega \bigr( h^{-1}A\varrho\cdot\varrho \ama_1+2i A \varrho \cdot D \ama_1 \bigr) \overline{(h\ama_2^*+r^*)} dx\cr
&\quad +\int_\Omega 
(AD\cdot hD(\ama_1+h\ama_2+r))\overline{(\ama_1^*+h\ama_2^*+r^*)}dx.\label{3.19}
\end{align}
We deduce from the Cauchy-Schwarz inequality and from the inequalities (\ref{3.12}), \eqref{3.12'}, \eqref{2.50} and \eqref{2.53} that
\begin{equation}\label{3.20}
\abs{\mathcal{J}_1(h)}\les (\|A \varrho \cdot \varrho\|_{L^\infty}+h)N_{2m, \varrho}(\ama_1)N_{2m, \overline{\varrho }}(\ama_1^*).
\end{equation}
In the same way as above, we can prove that
\begin{align}
h\int_\Omega(B\cdot Du+qu)\overline{u^*}dx &=-i\int_\Omega B\cdot\varrho \ama_1 \overline{\ama_1^*}dx + \mathcal{J}_2(h),\label{3.21}
\end{align}
where
\begin{align}\label{3.22}
   \mathcal{J}_2(h)&=-i\int_\Omega B\cdot\varrho \bigr(\ama_1 \overline{(h\ama_2^*+r^*)} + (h\ama_2+r)\overline{(\ama_1^*+h\ama_2^*+r^*)}\bigr)dx \cr
  &\quad+ \int_\Omega \bigr((hB\cdot D+hq)(\ama_1+h\ama_2+r)\bigr)\overline{(\ama_1^*+h\ama_2^*+r^*)}dx,
\end{align}
and satisfies
\begin{equation}
\abs{\mathcal{J}_2(h)}\les h N_{2m, \varrho}(\ama_1)N_{2m, \overline{\varrho }}(\ama_1^*).\label{3.23}
\end{equation}
Collecting (\ref{3.18}) and (\ref{3.21}), we find from (\ref{3.6}) that
\begin{align*}
&\bigr|\int_\Omega \bigr( h^{-1}A\varrho\cdot\varrho \ama_1+2 A \varrho \cdot \nabla \ama_1 +iB\cdot\varrho \ama_1 \bigr) \overline{\ama_1^*}dx\bigr|\cr &\les h\norm{(\Lambda^{(1)}-\Lambda^{(2)})f}_{\mathcal{H}^{m,\frac{3}{2}}}\norm{f^*}_{\mathcal{H}^{m,\frac{1}{2}}}
+(\|A \varrho \cdot \varrho\|_{L^\infty}+h) N_{2m, \varrho}(\ama_1)N_{2m, \overline{\varrho }}(\ama_1^*)\cr 
&\les h\norm{\Lambda^{(1)}-\Lambda^{(2)}}\norm{u}_{H^{2m}(\Omega)}\norm{u^*}_{H^{2m}(\Omega)}+ (\|A \varrho \cdot \varrho\|_{L^\infty}+h)N_{2m, \varrho}(\ama_1)N_{2m, \overline{\varrho }}(\ama_1^*) \cr
&\les (e^{C/h}\norm{\Lambda^{(1)}-\Lambda^{(2)}}+ h+\|A \varrho \cdot \varrho\|_{L^\infty})N_{2m, \varrho}(\ama_1)N_{2m, \overline{\varrho }}(\ama_1^*).
\end{align*}
Hence, the proof of Lemma \ref{L3.2} is completed.
\end{proof}

As a consequence of the above Lemma, we have the following integral identities.

\begin{lemma}\label{L3.3}
Let $A, B$ and $q$ as above. Then, for all $\ama_1,\,\ama_1^*\in C^\infty(\overline{\Omega})$ satisfy \eqref{3.13} and \eqref{3.14} respectively, we have
\begin{equation}\label{3.26}
\int_\Omega A \varrho \cdot \varrho \ama_1\overline{\ama_1^*}dx=\mathcal{I}_2(h),\quad \textrm{if}~~m>2,
\end{equation}
and 
\begin{equation}\label{3.25}
\int_\Omega (2 A \varrho \cdot \nabla \ama_1 + i \varrho \cdot B \ama_1)\overline{\ama_1^*}dx=\mathcal{I}_1(h),\quad \textrm{if}~~m=2,
\end{equation}
where 
\begin{align}\label{3.27}
 |  \mathcal{I}_j(h)| \les (e^{C/h} \|\Lambda^{(1)}-\Lambda^{(2)}\|+h)N_{2m, \varrho}(\ama_1)N_{2m, \overline{\varrho }}(\ama_1^*),
\end{align}
for $j=1,2$ and $m \geq 2$. Here $C$ is a positive constant independent of $h$. 
\end{lemma}

\begin{proof}
First, we prove \eqref{3.26}. Multiplying the equation \eqref{3.15} by $h$ we can obtain
\begin{align}\label{3.28}
    \int_\Omega  A \varrho \cdot \varrho \ama_1\overline{\ama_1^*}dx=-h\int_\Omega (2 A \varrho \cdot \nabla \ama_1 + i \varrho \cdot B \ama_1)\overline{\ama_1^*}dx+\mathcal{J}_2(h),
\end{align}
where $\mathcal{J}_2(h)$ satisfy \eqref{3.16}. Therefore by using the Cauchy-Schwarz inequality for the first term in the right hand side of \eqref{3.28}  we get our desired result \eqref{3.26}.
We move now to prove \eqref{3.27}. Since we assumed that $A$ is isotropic for $m=2$, then, $A \varrho \cdot \varrho=0$. Therefore, using Lemma \ref{L3.2} we get
\begin{align*}
 \int_\Omega (2 A \varrho \cdot \nabla \ama_1 + i \varrho \cdot B \ama_1)\overline{\ama_1^*}dx=\mathcal{J}_1(h),   
\end{align*}
where
\begin{align*}
    |  \mathcal{J}_1(h)| \les (e^{C/h} \|\Lambda^{(1)}-\Lambda^{(2)}\|+h)N_{4, {\varrho }}(\ama_1)N_{4, \overline{\varrho }}(\ama_1^*).
\end{align*}
 This completes the proof.
\end{proof}

We will now give the specific Hodge decomposition of symmetric tensor $A  \in H^k(\Omega, \mathbb{C}^{n^2})$ and a vector field $X\in W^{k, p}(\Omega, \mathbb{C}^n)$, with $k \geq 1$ and $p>n$, in the simply connected domain $\Omega \subset \R^n$.

\begin{lemma}\label{L4.1}
Let $k \geq 1$, any $2-$tensor field $A \in H^k(\Omega, \mathbb{C}^{n^2})$ can be uniquely represented as 
\begin{align}\label{4.1}
    A=A^\prime +\nabla_{\text{sym}}V + \vartheta_A~ id,
\end{align}
where $\vartheta_A \in  H^k(\Omega, \mathbb{C})$, the vector field $V
\in  H^{k+1}(\Omega, \mathbb{C}^{n})$ satisfies the boundary
condition
\begin{align}\label{4.2}
 V=0,\quad \textrm{on}~~ \Gamma,   
\end{align}
and the tensor field $A^\prime \in  H^k(\Omega, \mathbb{C}^{n^2})$ satisfies the conditions
\begin{align}\label{4.3}
 \operatorname{trace}(A^\prime)=0,\quad \sum_{j=1}^n \partial_j A^\prime_{jk}=0,~~ \forall k \in \{1, \ldots, n \}. 
\end{align}
Here $\nabla_{\text{sym}}$ is the symmetrized differentiation defined as $$(\nabla_{\text{sym}}V)_{ j,k}=\frac{1}{2}(\partial_{x_j}V_k+\partial_{x_k}V_j).$$
Moreover, the right hand side of \eqref{4.1} depends continuously on $A$ in the following sense : 
\begin{align}\label{4.4}
 \|A^\prime \|_{H^k(\Omega)}  \les  \|A\|_{H^k(\Omega)} ,\quad \|V\|_{H^{k+1}(\Omega)} \les  \|A\|_{H^k(\Omega)} \quad \textrm{and}~~
 \| \vartheta_A \|_{H^k(\Omega)} \les  \|A\|_{H^k(\Omega)}.
\end{align}
\end{lemma}
This Lemma is a straightforward adaptation of a Theorem $3.3$ in \cite{SV} and, for shake of completeness, we give the proof in the Appendix \ref{appendixA} .
\begin{lemma}\label{L4.9}
Let $k \in \mathbb{N}$ and $p >n$, any vector field $X \in W^{k, p}(\Omega, \mathbb{C}^n)$ can be represented as 
\begin{align}\label{4.80}
X=X^{\prime}+\nabla \vartheta_X,   
\end{align}
where $\vartheta_X \in W^{k+1, p}(\Omega, \C)$ satisfies the
boundary condition
\begin{align}\label{4.9}
    \vartheta_X=0,\quad \textrm{on}~~\Gamma,
\end{align}
and the vector field $X^{\prime} \in W^{k, p}(\Omega, \mathbb{C}^n)$ satisfies the condition
\begin{align}\label{4.100}
  \dive~X^\prime=0 ,\quad \textrm{in} ~~\Omega.  
\end{align}
Moreover, the right hand side of \eqref{4.80} depends continuously on 
$X$ in the following sense : 
\begin{align} \label{4.36}
   \|\vartheta_X\|_{W^{k+1, p}(\Omega)} \les \| X\|_{W^{k, p}(\Omega)},\quad \|X^{\prime}\|_{W^{1, p}(\Omega, \mathbb{C}^n)} \les \|\curl~X^{\prime}\|_{L^{p}(\Omega)}.
\end{align}
\end{lemma}
The extension by zero outside $\Omega$ of the $2-$tensor field $A^\prime$, the vector fields $V$ and $X^\prime$, and the functions $\vartheta_A$ and $\vartheta_X$ given by the 
previous Lemma is still denoted by the same letters.\medskip \\
The idea will be now to use the specific Hodge decomposition given 
by Lemmas \ref{L4.1} and \ref{L4.9} to decompose the symmetric tensor $A=A^{(2)}-A^{(1)} \in \mathcal{A}_{\sigma_1}\cap \mathcal{E}^\prime(\Omega)$ and the vector field $ B=B^{(2)}-B^{(1)} \in \mathcal{B}_{\sigma_2}\cap \mathcal{E}^\prime(\Omega)$ and write
\begin{align}
    A=A^\prime+\nabla_{\text{sym}}V+\vartheta_A~ id,\label{A}
\end{align}
where $A^\prime \in H^{\sigma_1}(\Omega)$ satisfies \eqref{4.3}, the vector field $V \in H^{\sigma_1+1}(\Omega)$ satisfies \eqref{4.2} and the function $\vartheta_A \in H^{\sigma_1}(\Omega)$. Moreover, we have 
\begin{align}\label{A1}
 \|A^\prime \|_{H^{\sigma_1}(\Omega)}  \les \|A\|_{H^{\sigma_1}(\Omega)} ,\quad \|V\|_{H^{\sigma_1}(\Omega)} \les \|A\|_{H^{\sigma_1}(\Omega)} \quad \textrm{and}~~
 \| \vartheta_A \|_{H^{\sigma_1}(\Omega)} \les \|A\|_{H^{\sigma_1}(\Omega)}.
\end{align}
In addition, since $\sigma_1 >\frac{n}{2}+3$~ then by using Sobolev's embedding one can easily see that $A^\prime \in W^{3, \infty}(\Omega, \C^{n^2})$, $V \in W^{4, \infty}(\Omega, \C^n)$ and $\vartheta_A \in W^{3, \infty}(\Omega, \C)$.\\
Further, the vector field $B$ can be decomposed by Lemma \ref{L4.9} as 
\begin{align}
 B=B^\prime+\nabla \vartheta_B,\label{B}   
\end{align}
where $B^\prime \in L^{ \infty}(\Omega, \mathbb{C}^{n})$ satisfies \eqref{4.100} and \eqref{4.36}. Hence, the Morrey's inequality yields
\begin{align}\label{B1}
    \|B^\prime\|_{L^\infty(\Omega)} \les \|\curl~B\|_{L^{\infty}(\Omega)}.
\end{align}
Moreover, the function $\vartheta_B \in W^{1, +\infty}(\Omega, \mathbb{C})$ solving the following boundary value problem
\begin{equation}
\left\{\begin{array}{ll}
\Delta \vartheta_B  = \dive~B & \mathrm{in}\, \Omega,\cr
\vartheta_B  =0  &\mathrm{on}\, \Gamma.
\end{array}
\right.
\end{equation}
Therefore, using the theorem on regular solvability of elliptic problem we have $\vartheta_B \in H^{\sigma_2+1}(\Omega)$ and satisfies the following estimate
\begin{align}\label{vartheta_B1}
     \| \vartheta_B\|_{H^{\sigma_2}(\Omega)} \les \|B\|_{H^{\sigma_2}(\Omega)}.
\end{align}

\label{section3}
\section{Stability estimates (first case \texorpdfstring{$m>2$}{Lg})} \label{section4}
This section is dedicated to proving the stable determination of
the symmetric tensor $A$, the vector field $B$ and the electric
potential $q$ from the Dirichlet to Neumann map $\Lambda_{A,B,q}$, 
in the case $m >2$.
We will use the family of solutions called complex geometric 
solutions, constructed in the previous section, to estimate 
the Fourier transform of the difference of the coefficients.\\
Consider a priori constant $M>0$, $\sigma_1 > \frac{n}{2}+3$ and $\sigma_2 > \frac{n}{2}+1$. Let $A^{(j)} \in \mathcal{A}_{\sigma_1}(M)\cap \mathcal{E}^\prime(\Omega)$, $B^{(j)} \in \mathcal{B}_{\sigma_2}(M)\cap \mathcal{E}^\prime(\Omega)$ and $q^{(j)}\in \mathcal{Q}_{}(M)$, $j=1,2$, be two sets of coefficients. We denote as section \ref{section3}
\begin{equation}
A=A^{(2)}-A^{(1)},\qquad B=B^{(2)}-B^{(1)}\quad\text{and  }~~~~ q=q^{(2)}-q^{(1)}.
\end{equation}
Moreover, without loss of generality, we will assume that
\begin{align}
    \|\Lambda^{(1)}-\Lambda^{(2)}\| \ll 1.
\end{align}

In order to prove the stability estimates, we need the following estimate.
\begin{lemma}\label{LA}
Let $k \in \N$ and $n_0 \geq 1$ be given. Let $\eta \in H^\sigma(\Omega)$, such that $\|\eta\|_{H^\sigma(\Omega)} \leq M$, for some 
$M>0$ and $ \sigma > \frac{n}{2}+k$, and let $\Lambda >0$. We assume that for any $0 <h \ll 1$ and $\xi \in \R^n$, we have
\begin{align}\label{0}
    \abs{\widehat{\widetilde{\eta} } (\xi )} \les \langle \xi \rangle^{n_0}(e^{C_0/h} \Lambda^{\mu_0} +h^{\kappa_0}),
\end{align}
for some $0< \mu_0 \leq 1$, $0< \kappa_0 \leq 1$ and $C_0>0$. Here $\widetilde{\eta}$ denotes the extension of $\eta$ outside $\Omega$  which satisfies $\tilde{\eta} \in H^\sigma(\R^n)$ and $\Vert \tilde{\eta}\Vert_{H^{\sigma}(\R^{n})}\les \,\Vert \eta\Vert_{H^{\sigma}(\Omega)}$. Then there exist $\mu, \kappa \in (0, 1)$ such that
\begin{align}\label{1}
    \| \eta \|_{W^{k, \infty}(\Omega)} \les e^{C/h} \Lambda^\mu +h^\kappa,
\end{align}
for any $h>0$ small enough. Here $C$ is positive constant independent of $h$ and $\kappa$ depends only on $n$ and $n_0$. Moreover, for any $n_1 \in \N$ there exist $\ell \gg 1$ and $\kappa_\ell >0$ such that 
\begin{align}\label{1*}
  h^{-n_1}  \| \eta \|_{W^{k, \infty}(\Omega)} \les e^{C/h^\ell} \Lambda^\mu +h^{\kappa_\ell}.
\end{align}

\end{lemma}

\begin{proof}
First, we may bound the $H^{-1}-$norm of $ \eta$. For $R >1$, to be chosen later, we can obtain 
from \eqref{0} the following inequality
\begin{align}\label{2}
 \| \eta\|_{H^{-1}(\Omega)}^2 &\leq    \| \widetilde{\eta}\|_{H^{-1}(\R^n)}^2=\int_{\langle \xi \rangle \leq R} \abs{\widehat{\widetilde{\eta}}(\xi)}^2 \langle \xi \rangle^{-2} d\xi+\int_{\langle \xi \rangle >  R} \abs{\widehat{\widetilde{\eta}}(\xi)}^2 \langle \xi \rangle^{-2} d\xi\cr
 &\les R^{2n_0+n-2}(e^{{2C_0}/{h}} \Lambda^{2\mu_0} +h^{\kappa_0})+M^2 R^{-2}.
\end{align}
Choosing $R>1$  such that $R^{2n_0+n-2} h^{\kappa_0}=R^{-2}$ that is $R=h^{\frac{- \kappa_0}{n+2n_0}}$ then we deduce from \eqref{2} for $h_0$ sufficiently small that
\begin{align*}
   \|\eta\|_{H^{-1}(\Omega)} \les  e^{C^\prime/h} \Lambda^{\mu_0} + h^{\frac{\kappa_0}{n+2n_0}},
\end{align*}
for some positive constant $C^\prime$ and all $h < h_0$. Here we used the fact that $h^{-a} \leq e^{C^\prime/h}$, for some $0<a \leq 1$.\\
In order to complete the proof of the theorem, let 
$\delta=\frac{1}{2}(\sigma-(\frac{n}{2}+\abs{\ama}))>0$, where $\ama \in \N^n$ such that $\abs{\ama} \leq k$, then using Sobolev's 
embedding theorem together with interpolation theorem, we end up getting the following inequality 
\begin{align*}
 \|D^\alpha \eta\|_{L^{\infty}(\Omega)} & \les \|D^\alpha \eta\|_{H^{n/2+\delta }(\Omega)} \\
& \les \|D^\alpha \eta\|_{H^{-1-\abs{\ama}}(\Omega)}^{\kappa_1}\|D^\alpha \eta \|_{H^{n/2+2 \delta }(\Omega)}^{ 1-\kappa_1} \\
& \les e^{{ \kappa_1 C^\prime}/{h}} \Lambda^{\mu_0 \kappa_1} + h^{\frac{\kappa_0\kappa_1}{n+2n_0}},
\end{align*}
for some $\kappa_1 \in (0,1)$. This concludes the proof of \eqref{1} with $\mu=\mu_0 \kappa_1$ and $\kappa=\frac{\kappa_0\kappa_1}{n+2n_0}$.\\
We move now to prove \eqref{1*} we observe that \eqref{1} implies in particular 
\begin{align*}
    \| \eta \|_{W^{k, \infty}(\Omega)} \les e^{C/h^\ell} \Lambda^\mu +h^{\ell\kappa},
\end{align*}
for all $\ell \gg 1$. Choosing $\ell$ such that $\kappa_\ell:=\ell \kappa -n_1 >0$, we obtain 
\begin{align*}
  h^{-n_1}  \| \eta \|_{W^{k, \infty}(\Omega)} \les e^{C/h^\ell} \Lambda^\mu +h^{\kappa_\ell},
\end{align*}
with possibly different constant $C$. This completes the proof.
\end{proof}

As a corollary of Lemma \ref{L3.3} and the decomposition \eqref{A}, we have the following integral identity.
\begin{lemma}\label{L4.2}
Let $A^\prime \in W^{3, +\infty}(\Omega, \mathbb{C}^{n^2})$ and $V
\in W^{4, +\infty}(\Omega, \mathbb{C}^{n})$ given by \eqref{A}. Then, for all $\ama_1,\,\ama_1^*\in C^\infty(\overline{\Omega})$ satisfy \eqref{3.13} and \eqref{3.14} respectively, we have
\begin{align}\label{4.5}
  \int_\Omega (A^\prime \varrho \cdot \varrho) \ama_1\overline{\ama_1^*}dx-\int_\Omega (\varrho \cdot V) \varrho \cdot \nabla (\ama_1\overline{\ama_1^*})dx=\mathcal{I}(h), 
\end{align}
where 
\begin{align}\label{4.6}
 |  \mathcal{I}(h)| \les (e^{C/h} \|\Lambda^{(1)}-\Lambda^{(2)}\|+h)N_{2m, \varrho}(\ama_1)N_{2m, \overline{\varrho }}(\ama_1^*),
\end{align}
for any $h>0$ small enough. Here $C$ is a positive constant independent of $h$. 
\end{lemma}
\begin{proof}
Substituting $A$ into the left hand side of \eqref{3.26}, we get 
\begin{align}\label{4.7}
     \int_\Omega (A^\prime \varrho \cdot \varrho) \ama_1\overline{\ama_1^*}dx+\int_\Omega (\nabla_{\text{sym}}V \varrho \cdot \varrho) \ama_1\overline{\ama_1^*}dx=\mathcal{I}_2(h),
\end{align}
where $\mathcal{I}_2(h)$ satisfy \eqref{3.27}. Here we used the fact that $\varrho \cdot \varrho=0$.\\
Now, using an integration by part for the second term in the left hand side of \eqref{4.7} and taking into account the condition \eqref{4.2}, we get our desired estimate.
\end{proof}

With the help of the previous lemma, we are now in position to derive an estimate for the $2-$tensor $A^\prime$ as follow.


\begin{lemma}\label{L4.5}
Let $A^\prime$ as above. There exist $C_1>0$ and $\mu_1, \kappa_1 \in (0,1)$, such that we have
\begin{align}\label{4.24}
     \|A^\prime\|_{L^{\infty}(\Omega)} \les  e^{C_1/h} \|\Lambda^{(1)}-\Lambda^{(2)}\|^{\mu_1} +h^{\kappa_1},
\end{align}
for any $h >0$ small enough. Here $C_1$ is independent of $h$ and $\kappa_1$ depends only on $n$ and $m$.
\end{lemma}
\begin{proof}
Let $\omega, \widetilde{\omega} \in \mathbb{S}^{n-1}$ and $\xi \in \mathbb{R}^n$ such that $\xi \cdot \omega =\xi \cdot \widetilde{\omega }=\omega \cdot \widetilde{\omega}=0$. Choosing $\ama_1 (x)=e^{-i x\cdot \xi}$ and $\ama_1^*(x)=1$,~~for $x \in \Omega$. It is clear that $\ama_1 $ and $\ama_1^*$ satisfy, respectively, the transport equations \eqref{3.13} and \eqref{3.14}. Then, by using Lemma \ref{L4.2} and the fact that $\varrho \cdot \xi=0$, we obtain
\begin{align}\label{4.8}
  \abs{\int_{\Omega} e^{-i x\cdot \xi} {A^\prime}(x)\varrho \cdot \varrho \mathrm{~d} x}  \les \langle\xi \rangle^{2m}\bigr(e^{C/h }\|\Lambda^{(1)}-\Lambda^{(2)}\|+h\bigr). 
\end{align}
Now, let us fix $\xi \in \mathbb{R}^n\setminus \{ 0\}$. Following \cite{B2}, we consider the orthonormal basis $\mathscr{B}$ of $\mathbb{R}^n$ as
\begin{align*}
   \mathscr{B}=\{ \omega_1, \omega_2, \ldots, \omega_n\} \qquad \textrm{with} \quad \omega_n=\frac{\xi}{\abs{\xi }}.
\end{align*}
The inequality \eqref{4.8} is valid for all $\omega, \widetilde{\omega} \in \mathbb{S}^{n-1}$ such that $\omega \cdot \xi= \widetilde{\omega} \cdot \xi =0$. Then, for $\varrho=\omega_j+i \omega_k$, where $j, k \in \{1, \ldots, n-1\}$, $j \neq k$, we have
\begin{align}\label{4.10}
    \abs{\widehat{A^\prime}(\xi) (\omega_j+i \omega_k)\cdot(\omega_j+i\omega_k)  }  \les \langle\xi \rangle^{2m}\bigr(e^{C/h }\|\Lambda^{(1)}-\Lambda^{(2)}\|+h\bigr). 
\end{align}
Replacing $\omega_k$ by $-\omega_k$ in $\mathscr{B}$ and doing the same analysis as before we get
\begin{align}\label{4.11}
    \abs{\widehat{A^\prime}(\xi) (\omega_j-i \omega_k)\cdot(\omega_j-i\omega_k) }  \les \langle\xi \rangle^{2m}\bigr(e^{C/h }\|\Lambda^{(1)}-\Lambda^{(2)}\|+h\bigr),
\end{align}
for all $j,k\in\{1, \ldots, n-1\}, ~j \neq k.$ Combining \eqref{4.10} and \eqref{4.11} together we obtain
\begin{align}\label{4.12}
| \widehat{A^\prime}(\xi)\omega_j \cdot \omega_j - \widehat{A^\prime}(\xi)\omega_k \cdot \omega_k |\les\langle\xi \rangle^{2m}\bigr(e^{C/h }\|\Lambda^{(1)}-\Lambda^{(2)}\|+h\bigr), 
\end{align}
for all $j,k\in\{1, \ldots, n-1\}, ~j \neq k.$ Moreover, their difference gives us 
\begin{align}\label{4.13}
| \widehat{A^\prime}(\xi)\omega_j \cdot \omega_k  |\les\langle\xi \rangle^{2m}\bigr(e^{C/h }\|\Lambda^{(1)}-\Lambda^{(2)}\|+h\bigr),
\end{align} 
for all $ j,k\in\{1, \ldots, n-1\}, ~j \neq k.$ In addition, since $A^\prime$ is 
divergence free we get $\sum_{k=1}^{n} \widehat{A^\prime}_{jk}(\xi)\xi_k=0$, for 
any $j\in \{1, \ldots, n\}$, then, due to the fact that $A^\prime$ is symmetric we can obtain that 
\begin{align}\label{4.14}
    \widehat{A^\prime}(\xi)\omega_j \cdot \xi=0, \quad  j = {1, \ldots, n}.
\end{align}
Which shows that the vector $\widehat{A^\prime}(\xi)\omega_j$ can be written with respect to orthonormal basis $\mathscr{B}$ as
\begin{align}\label{4.15}
    \widehat{A^\prime}(\xi)\omega_j=\sum_{k=1}^{n-1}\alpha_k(\xi)\omega_k\quad j={1, \ldots, n},
\end{align}
where $\alpha_k(\xi) \in \mathbb{C}$. Let us write $\lambda_{j, k}(\xi):=\widehat{A^\prime}(\xi)\omega_j \cdot \omega_k,~\textrm{where}~j, k\in\{1, \ldots, n-1\}.$
Then it follows from \eqref{4.12} and \eqref{4.13} that
\begin{align}
  | \lambda_{j, j}(\xi) - \lambda_{k, k}(\xi)  |&\les \langle\xi \rangle^{2m}\bigr(e^{C/h }\|\Lambda^{(1)}-\Lambda^{(2)}\|+h\bigr), \label{4.16}   
\end{align}
and for $j \neq k$, we have 
\begin{align}\label{4.17}
 | \lambda_{j, k}(\xi)  |&\les \langle\xi \rangle^{2m}\bigr(e^{C/h }\|\Lambda^{(1)}-\Lambda^{(2)}\|+h\bigr).
\end{align}
Let us define the symmetric $n \times n$ matrix $M(\xi)$ given by
\begin{align*}
M(\xi)=\begin{pmatrix}
\lambda_{1,1}(\xi) & \ldots & \lambda_{1,n-1}(\xi) & 0 \\
\vdots  & \ddots  & \vdots  & \vdots  \\
\lambda_{n-1,1}(\xi) & \ldots  & \lambda_{n-1,n-1}(\xi) &  0\\
0 & \ldots  & 0 & 0 \\
\end{pmatrix}.
\end{align*}
Then, we can write $ \widehat{A^\prime}(\xi)={}^t\mathscr{B} M(\xi) \mathscr{B}, $ where $^t \mathscr{B}$ is the transpose of the matrix $\mathscr{B}$. We observe that $M=D+M^\prime$, where $D$ is the diagonal matrix given as $$D(\xi)=\operatorname{Diag}\bigr(\lambda_{1,1}(\xi), \ldots, \lambda_{n-1, n-1}(\xi),0\bigr),$$ and $M^\prime=M-D$. Since $\operatorname{trace}(A^\prime)=0$ then $\operatorname{trace}(D)=0$ which implies, for any $j=1, \ldots,n-1$, that 
\begin{align*}
 \lambda_{j,j}(\xi)&=-\sum_{\substack{k=1 \\{k \neq j}}}^{n-1} \lambda _{k,k}(\xi)=-\sum_{\substack{k=1 \\{k \neq j}}}^{n-1} \bigr(\lambda_{k,k}-\lambda_{j,j}\bigr)(\xi)-(n-2)\lambda_{j,j}(\xi).  
\end{align*}
Then, using \eqref{4.16} we get, for all $j=1, \ldots n-1$
\begin{align}\label{4.18}
    | \lambda_{j,j}(\xi)| \les \langle\xi \rangle^{2m}\bigr(e^{C/h }\|\Lambda^{(1)}-\Lambda^{(2)}\|+h\bigr). 
\end{align}
On the other hand, it follows from \eqref{4.17} that
\begin{align}\label{4.19}
    | M^\prime_{j,k}(\xi)| \les \langle\xi \rangle^{2m}\bigr(e^{C/h }\|\Lambda^{(1)}-\Lambda^{(2)}\|+h\bigr). 
\end{align}
Thus, from \eqref{4.18} and \eqref{4.19} one can directly conclude that
\begin{align}\label{4.20}
   |\widehat{A^\prime}(\xi) |\les  \langle\xi \rangle^{2m}\bigr(e^{C/h }\|\Lambda^{(1)}-\Lambda^{(2)}\|+h\bigr),
\end{align}
here we used the fact that $|\omega_j|=1$, $j={1, \ldots, n}$.\\
However, since $A^\prime$ satisfy the inequality \eqref{A1} then by using Lemma \ref{LA} we obtain our desired estimate.
\end{proof}

The next step is to estimate the vector field $V \in W^{4, \infty}(\Omega, \mathbb{C}^n)$, given by decomposition \eqref{A}. To do this, we use again the Hodge decomposition for a vector fields given by Lemma \ref{L4.9} and we write
\begin{align}
   V&=V^\prime+\nabla \vartheta_V,\label{V}  
\end{align}
where $V^{\prime} \in W^{3, \infty}(\Omega, \mathbb{C}^n)$ satisfies \eqref{4.100} and 
\begin{align}\label{V1}
    \|V^\prime\|_{L^{\infty}(\Omega)} \les \|\curl~V\|_{L^{\infty}(\Omega)},
\end{align}
Moreover, the function $\vartheta_V \in W^{4, \infty}(\Omega, \C)$ satisfies the boundary condition $\vartheta_V=0$, on $\Gamma$, and
\begin{align}\label{vartheta_V1}
     \| \vartheta_V\|_{H^{\sigma_1}(\Omega)} \les \|V\|_{H^{\sigma_1}(\Omega)}.
\end{align}
Here $$\curl~V=\sum_{j, k=1}^{n} \mathrm{v}_{jk} d x_{j} \wedge d x_{k},$$
with  $\mathrm{v}_{jk}(x)=\partial_{x_j}V_k-\partial _{x_k}V_j , ~~j, k=1, \ldots, n$
and $V_{k}(x)=V(x) \cdot e_{k},~ x \in \Omega,$ where $\left(e_{k}\right)_{k}$ is the canonical basis of $\mathbb{R}^{n}$.
\smallskip

The $L^\infty-$norm of the vector field $V^\prime$ can be estimate as follow.

\begin{lemma}\label{L4.7}
Let $V^\prime$ given by \eqref{V}. There exist $C_2 >0$ and $\mu_2, \kappa_2 \in (0, 1)$ such that we have the following estimate
  \begin{align}
  \|V^\prime\|_{W^{1, \infty}(\Omega)} \les e^{{C_2}/{h}} \|\Lambda^{(1)}-\Lambda^{(2)}\|^{\mu_2} +h^{\kappa_2},
      \label{4.30}
  \end{align}
for any $h >0$ small enough. Here $C_2$ is a positive constant independent of $h$ and $\kappa_2 $ depends only on $m$ and $n$.
\end{lemma}

\begin{proof}
First, we start by estimate the Fourier transform of $\curl~V$. Let $\xi \in \mathbb{R}^{n}$. We select $\omega, \widetilde{\omega} \in \mathbb{S}^{n-1}$ such that $\xi, \omega$ and $\widetilde{\omega}$ be three mutually orthogonal vectors in $\R^{n}$. Choosing $\ama_1 (x)=e^{-i x\cdot \xi}(x \cdot \omega)$ and $\ama_1^*(x)=1$,~~for $x \in \Omega$. It is clear that $\ama_1 $ and $\ama_1^*$ satisfy respectively \eqref{3.13} and \eqref{3.14}. Then, by using Lemma \ref{L4.2} and the fact that $\varrho \cdot \xi=0$, we obtain 
\begin{align}\label{4.25}
  \abs{ \varrho \cdot \widehat{V}(\xi)}  \les \langle\xi \rangle^{2m}\bigr( \|A^\prime\|_{L^{\infty}(\Omega)} +e^{C/h }\|\Lambda^{(1)}-\Lambda^{(2)}\|+h\bigr).  
\end{align}
In addition Lemma \ref{L4.5} yields
\begin{align}\label{4.27}
    \abs{ \varrho \cdot \widehat{V}(\xi)} \les \langle \xi\rangle^{2m} \bigr(e^{C^\prime/h}\|\Lambda^{(1)}-\Lambda^{(2)}\|^{\mu_1}+h^{\kappa_1} \bigr),
\end{align}
for some constant $C^\prime=\max(C, C_1) > 0$. Replacing $\widetilde{\omega }$ by $-\widetilde{\omega }$ and doing the same analysis as before, we get 
\begin{align}\label{4.28}
   \abs{ \overline{\varrho} \cdot \widehat{V}(\xi)} \les \langle \xi\rangle^{2m} \bigr(e^{C^\prime/h}\|\Lambda^{(1)}-\Lambda^{(2)}\|^{\mu_1}+h^{\kappa_1} \bigr).
\end{align}
Combining the inequalities \eqref{4.27} and \eqref{4.28} together, we find that
\begin{align}
    \ \abs{ \omega \cdot \widehat{V}(\xi)} \les \langle \xi\rangle^{2m} \bigr(e^{C^\prime/h}\|\Lambda^{(1)}-\Lambda^{(2)}\|^{\mu_1}+h^{\kappa_1} \bigr).
    \label{4.29}
\end{align}
For any fixed $\xi \in \mathbb{R}^n$ with $\xi \cdot \omega=0$. Let us choose $\omega =\frac{\xi_{j} e_{k}-\xi_{k} e_{j}}{|\xi_{j} e_{k}-\xi_{k} e_{j}|}$, for $j \neq k \in \{ 1, \ldots, n\}$,  then by multiplying \eqref{4.29} by $\abs{\xi_{j} e_{k}-\xi_{k} e_{j}}$, we obtain
\begin{align}\label{4.26}
\abs{\widehat{\mathrm{v}}_{jk}(\xi)} \les \langle \xi\rangle^{2m+1} \bigr(e^{C^\prime/h}\|\Lambda^{(1)}-\Lambda^{(2)}\|^{\mu_1}+h^{\kappa_1} \bigr).
\end{align}
Hence, from Lemma \ref{LA} and the inequality \eqref{V1}, there exist $C > 0$ and $\mu, \kappa=\kappa(n, m) \in (0,1)$ such that we have 
\begin{align}
   \|V^\prime\|_{L^{\infty}(\Omega)} \les  \|\curl~V\|_{L^{\infty}(\Omega)} \les e^{C/h} \|\Lambda^{(1)}-\Lambda^{(2)}\|^{\mu} +h^{\kappa}.\label{4.31'}
\end{align}
In order to complete the proof of the theorem, let $\delta=\frac{1}{2}(\sigma_1-(\frac{n}{2}+1))>0$, then using Sobolev's embedding theorem together with interpolation theorem and Plancherel theorem, we get for $\kappa_0 \in (0,1)$
\begin{align*}
    \|DV^\prime\|_{L^{\infty}(\Omega)} & \les \|DV^\prime\|_{H^{n/2+\delta }(\Omega)} \\
& \les \|DV^\prime\|_{H^{-1}(\Omega)}^{\kappa_0}\|DV^\prime\|_{H^{n/2+2 \delta }(\Omega)}^{ 1-\kappa_0} \\
 &\les \|V^\prime\|_{L^\infty(\Omega)}^{\kappa_0}.
\end{align*}
This and the inequality \eqref{4.31'} conclude the proof.
\end{proof}

With the help of the above Lemma, we may now prove the $W^{1, \infty}-$norm of $V$.
\begin{lemma}\label{L4.11}
Let $V$ given by \eqref{V}. There exist $C_3>0$ and $\mu_3$, $\kappa_3 \in (0,1)$ such that the following estimate
\begin{align}\label{4.40}
   \|V\|_{W^{1, \infty}(\Omega)}&\les e^{C_3/h} \|\Lambda^{(1)}-\Lambda^{(2)}\|^{\mu_3} +h^{\kappa_3},
\end{align}
holds for any $h >0$ small enough. Here $C_3$ is independent of $h$ and $\kappa_3$ depends on $n$ and $m$.
\end{lemma}

\begin{proof}
In the first step, we will estimate $\widehat{\vartheta}_V$ given by \eqref{V}. Let $\xi \in \R^n$, we choose $\omega, \widetilde{\omega} \in \mathbb{S}^{n-1}$ such that $\xi, \omega$ and $\widetilde{\omega}$ be three mutually orthogonal vectors in $\R^{n}$. By substituting $V$ into the left hand side of \eqref{4.5}, we get for $\ama_1(x)=e^{-i x \cdot \xi}$
\begin{equation}\label{4.38}
\int_\Omega e^{-i x \cdot \xi} (\varrho \cdot \nabla \vartheta_V)  (\varrho \cdot \nabla \overline{\ama_1^*}) ~dx=-\int_\Omega e^{-i x \cdot \xi} (\varrho \cdot V^{\prime}) (\varrho \cdot \nabla \overline{\ama_1^*})~dx + \mathcal{I}_3(h),
\end{equation}
where $\mathcal{I}_3$ satisfies 
\begin{align}\label{4.39}
  |  \mathcal{I}_3(h)| \les  \langle \xi\rangle^{2m}(\|A^\prime\|_{L^\infty}+e^{C/h} \|\Lambda^{(1)}-\Lambda^{(2)}\|+h)N_{2m, \overline{\varrho }}(\ama_1^*).
\end{align}
So using integration by parts for the left hand side of the above equality, we obtain
\begin{align*}
- \int_{\Omega} e^{-i x \cdot \xi} \vartheta_V (\varrho \cdot \nabla)^2 \overline{ \ama_1^*}\mathrm{~d} x =- \int_\Omega e^{-i x \cdot \xi} (\varrho \cdot V^{\prime}) (\varrho \cdot \nabla \overline{\ama_1^*})~dx + \mathcal{I}_3(h),  
\end{align*}
here we used the fact that $ \vartheta_V =0$, on $\Gamma $, and $\omega \cdot \xi=\widetilde{\omega} \cdot \xi=0$. We consider now $\ama_1^*(x)=-\frac{1}{2}(\omega  \cdot x)^2$, $x \in \Omega$, which satisfies \eqref{3.14}.
Thus, one can show that
\begin{align*}
    \abs{ \widehat{\vartheta}_V(\xi)}
\les   \langle \xi\rangle^{2m}(\|V^\prime \|_{L^{\infty}(\Omega)}+\|A^\prime\|_{L^\infty}+e^{C/h} \|\Lambda^{(1)}-\Lambda^{(2)}\|+h).
\end{align*}
Then from Lemmas \ref{L4.5} and \ref{L4.7}, we obtain
\begin{align*}
\abs{\widehat{\vartheta}_V(\xi)}
\les  \langle \xi\rangle^{2m} (  e^{C^\prime/h}\|\Lambda^{(1)}-\Lambda^{(2)}\|^{\theta_1}+h^{\theta_2}),
\end{align*}
where $C^\prime=\max(C, C_1, C_2)$ is a positive constant independent of $h$,  $\theta_1 =\min(\mu_1, \mu_2) \in (0,1)$ depends on $n$ and $m$, and $\theta_2=\min(\kappa_1, \kappa_2 ) \in (0,1)$.\\
Moreover, as $\vartheta_V$ verifies \eqref{vartheta_V1} then by applying Lemma \ref{LA}, we get 
\begin{align}\label{varth_V}
    \|\vartheta_V\|_{W^{2, \infty}(\Omega)}&\les e^{C_0/h} \|\Lambda^{(1)}-\Lambda^{(2)}\|^{\mu_0} +h^{\kappa_0},
\end{align}
for some positive constant $C_0$ and $\mu_0$, $\kappa_0 \in (0,1)$.\\ Taking into account \eqref{V}, we may conclude from Lemma \ref{L4.7} and \eqref{varth_V} that Lemma \ref{L4.11} is completely proved with $C_3=\max(C_2, C_0)$ which is a positive constant independent of $h$, $\mu_3 =\min(\mu_2, \mu_0)\in (0,1)$ and $\kappa_3=\min(\kappa_2, \kappa_0 )\in (0,1)$ which depends on $n$ and $m$.
\end{proof}
In order to estimate the function $\vartheta_A$ given by \eqref{A}, we start by the following integral identity.
 \begin{lemma}\label{L4.12}
 Let $B \in \mathcal{B}_{\sigma_2}(M) \cap \mathcal{E}^\prime(\Omega)$ and $\vartheta_A$ given by \eqref{A}. There exist $\ell_1 \gg 1$ and $\mu_4$, $\kappa_4 \in (0,1)$ such that for all $\ama_1,\,\ama_1^*\in C^\infty(\overline{\Omega})$ satisfy \eqref{3.13} and \eqref{3.14} respectively, we have
\begin{equation}\label{4.41}
\int_\Omega 2 \vartheta_A (\varrho \cdot \nabla \ama_1)\overline{\ama_1^*}dx + i\int_\Omega (\varrho \cdot B) \ama_1\overline{\ama_1^*}dx=\mathcal{J}(h),
\end{equation}
where 
\begin{align}\label{4.42}
 |  \mathcal{J}(h)| \les \bigr( e^{C_4/{h^{\ell_1}}} \|\Lambda^{(1)}-\Lambda^{(2)}\|^{\mu_4}+h^{\kappa_4})N_{2m, \varrho}(\ama_1)N_{2m, \overline{\varrho }}(\ama_1^*).
\end{align}
Here $C_4$ is a positive constant independent of $h$ and $\kappa_4, \ell_1$ depend only on $n$ and $m$. 
 \end{lemma}

\begin{proof}
Substituting $A$ in Lemma \ref{L3.2} we get 
\begin{align*}
    &\int_\Omega 2 \vartheta_A (\varrho \cdot \nabla \ama_1)\overline{\ama_1^*}dx + i\int_\Omega (\varrho \cdot B) \ama_1\overline{\ama_1^*}dx\\&=-\int_\Omega \bigr(h^{-1} A^\prime \varrho \cdot \varrho \ama_1+h^{-1} \nabla_{\text{sym}}V \varrho \cdot \varrho \ama_1+2 A^\prime \varrho \cdot \nabla \ama_1+2 \nabla_{\text{sym}}V \varrho \cdot \nabla \ama_1 \bigr)\overline{\ama_1^*}dx+\mathcal{I}(h)\cr
    &:=\mathcal{J}(h),
\end{align*}
where $\mathcal{I}(h)$ satisfies 
\begin{align*}
\abs{ \mathcal{I}(h)}
  &\les (e^{C/h} \|\Lambda^{(1)}-\Lambda^{(2)}\|+h+\|A^\prime \|_{L^\infty}+\|V\|_{W^{1, \infty}})N_{2m, \varrho}(\ama_1)N_{2m, \overline{\varrho }}(\ama_1^*).
\end{align*}
Then by applying the Cauchy-Schwarz inequality for the first term in the right hand side of the above expression we obtain
\begin{align}\label{J0}
\abs{ \mathcal{J}(h)}
  &\les \bigr( h^{-1} (\|A^\prime \|_{L^\infty}+\|V\|_{W^{1, \infty}})+e^{C/h} \|\Lambda^{(1)}-\Lambda^{(2)}\|+h\bigr)N_{2m, \varrho}(\ama_1)N_{2m, \overline{\varrho }}(\ama_1^*).
\end{align}
Using now Lemmas \ref{L4.5} and \ref{L4.11}, then with the help of the last part of Lemma \ref{LA}, there exist $\widetilde{\ell}_1, \widetilde{\ell}_4 \gg1$ and $\widetilde{\kappa}_1, \widetilde{\kappa}_4 \in (0,1)$ such that we have 
\begin{align*}
    \abs{ \mathcal{J}(h)}&\les \bigr(e^{C/{h^{\widetilde{\ell}_1}}} \|\Lambda^{(1)}-\Lambda^{(2)}\|^{\mu_1} +h^{\widetilde{\kappa}_1}+e^{C/{h^{\widetilde{\ell}_4}}} \|\Lambda^{(1)}-\Lambda^{(2)}\|^{\mu_3} +h^{\widetilde{\kappa}_4}\bigr)N_{2m, \varrho}(\ama_1)N_{2m, \overline{\varrho }}(\ama_1^*).
\end{align*}
This completes the proof with $\ell_1=\max(\widetilde{\ell}_1, \widetilde{\ell}_4)\gg 1$, $\mu_4=\min(\mu_1, \mu_3) \in (0,1)$ and $\kappa_4=\min(\widetilde{\kappa}_1, \widetilde{\kappa}_4) \in (0,1)$.
\end{proof}


\subsection{Stability estimate for the vector field (first case \texorpdfstring{$m>2$}{Lg}) }
We derive in this section a stability estimate for the vector field $B$. First, we will use the Hodge decomposition given by \eqref{B} and the integral identity \eqref{4.41} to estimate the Fourier transform of $\curl~ B$. Second, we exploit the boundness of $B$ to prove the stability estimate for $B$ it self.
\smallskip

In what follows, we denote for $B=(B_1, \ldots, B_n)$
\begin{align*}
   \mathrm{b}_{jk}=\partial_{x_j}B_k-\partial _{x_k}B_j ,~~ j, k=1, \ldots, n,
\end{align*}
the components of $\curl~B$ and $\widehat{\mathrm{b}}_{jk}$ the associated Fourier coefficients.\\
Since the functions $\ama_1$ and $\ama_1^*$ are arbitrary solutions of \eqref{3.13} and \eqref{3.14}, respectively, our strategy is to choose suitable solutions $\ama_1$ and $\ama_1^*$ such that, in the first step, the first term in the left hand side of \eqref{4.41} becomes zero in order to obtain an estimate of $\curl~B$. Second, we return to identity \eqref{4.41} and by choosing again a suitable solutions $\ama_1$ and $\ama_1^*$ we may estimate the Fourier transform of $\vartheta_A$ in terms of norm of $B$. We then have the following estimate of $\curl~B$.
\begin{lemma}\label{L4.14}
Let $B$ as above. There exist $\ell_2 \gg 1$ and $\mu_5, \kappa_5 \in (0, 1)$ such that we have the following estimate
  \begin{align}
  \|\curl~B\|_{L^{\infty}(\Omega)} \les e^{C_5/{h^{\ell_2}}} \|\Lambda^{(1)}-\Lambda^{(2)}\|^{\mu_5 } +h^{\kappa_5},
      \label{4.44}
  \end{align}
for any $h >0$ small enough. Here $C_5$ is a positive constant independent of $h$ and $\kappa_5 $, $\ell_2$ depend only on $n$ and $m$.
\end{lemma}

\begin{proof}
Let $\xi \in \mathbb{R}^{n}$. We select $\omega, \widetilde{\omega} \in \mathbb{S}^{n-1}$ such that $\xi, \omega$ and $\widetilde{\omega}$ be three mutually orthogonal vectors in $\R^{n}$. Let us choose $\ama_1 (x)=e^{-i x \cdot \xi}$ and $\ama_1^*(x)=1$, $x \in \Omega$. Then $\ama_1 $ and $\ama_1^*$ satisfy respectively \eqref{3.13} and \eqref{3.14}. Using now Lemma \ref{L4.12} and the fact that $\varrho \cdot \xi=0$, we obtain 
\begin{align*}
    &\abs{\varrho \cdot \widehat{B}(\xi)}\les \langle \xi\rangle^{2m}\bigr(e^{C_4/{h^{\ell_1}}}\norm{\Lambda^{(1)}-\Lambda^{(2)}}^{\mu_4 }+h^{\kappa_4 }).
\end{align*}
Then by doing the same analysis as the first part of the proof of Lemma \ref{L4.7}, we get
\begin{align*}
    &\abs{\widehat{\mathrm{b}}_{jk}(\xi)}\les \langle \xi\rangle^{2m+1}\bigr(e^{C_4/{h^{\ell_1}}}\norm{\Lambda^{(1)}-\Lambda^{(2)}}^{\mu_4 }+h^{\kappa_4 }).
\end{align*}
Therefore, by Lemma \ref{LA} there exist $C_5 >0$ and $\mu_5, \kappa_5 \in (0, 1)$ such that we have 
  \begin{align}
  \|\curl~B\|_{L^{\infty}(\Omega)} \les e^{C_5/{h^{\ell_1}}} \|\Lambda^{(1)}-\Lambda^{(2)}\|^{\mu_5 } +h^{\kappa_5}.
  \end{align}
This completes the proof.
\end{proof}

The next step is to estimate the $W^{1, \infty}-$norm of $\vartheta_B$.

\begin{lemma}\label{L4.15}
Let $\vartheta_B$ given by \eqref{B}. There exist $\ell_3 \gg 1$ and $\mu_6$, $\kappa_6 \in (0,1)$ such that we have
\begin{align}\label{4.49}
   \|\vartheta_B\|_{W^{1, \infty}(\Omega)}&\leq e^{C_6/{h^{\ell_3}}} \|\Lambda^{(1)}-\Lambda^{(2)}\|^{\mu_6} +h^{\kappa_6},
\end{align}
for any $h >0$ small enough. Here $C_6$ is a positive constant independent of $h$ and $\kappa_6, \ell_3$ depend on $n$ and $m$.
\end{lemma}

\begin{proof}
 We start by estimate the Fourier transform of $\vartheta_B$. Let $\xi \in \R^n$, we choose $\omega, \widetilde{\omega} \in \mathbb{S}^{n-1}$ such that $\xi, \omega$ and $\widetilde{\omega}$ be three mutually orthogonal vectors in $\R^{n}$. Substituting $B$  into the left hand side of \eqref{4.41} then we get for $\ama_1(x)=e^{-i x \cdot \xi}$
\begin{equation}\label{4.50}
\int_\Omega e^{-i x \cdot \xi}(\varrho \cdot \nabla \vartheta_B)  \overline{\ama_1^*} ~dx=-\int_\Omega e^{-i x \cdot \xi} (\varrho \cdot B^{\prime}) \overline{\ama_1^*}~dx + \mathcal{I}(h),
\end{equation}
where $\mathcal{I}(h)$ satisfies 
\begin{align}\label{4.51}
  |  \mathcal{I}(h)| \les  \langle \xi\rangle^{2m} (e^{C_4/h^{\ell_1}} \|\Lambda^{(1)}-\Lambda^{(2)}\|^{\mu_4}+h^{\kappa_4})N_{2m, \overline{\varrho }}(\ama_1^*).
\end{align}
So using integration by parts for the left hand side of the above equality, we obtain
\begin{align*}
- \int_{\Omega} e^{-i x \cdot \xi}  \vartheta_B (\varrho \cdot \nabla \overline{ \ama_1^*})\mathrm{~d} x =- \int_{\Omega} e^{-i x \cdot \xi} (\varrho \cdot B^\prime) \overline{\ama_1^*}\mathrm{~d} x+ \mathcal{I}(h),  
\end{align*}
here we used the fact that $\vartheta_B =0$, on $\Gamma $ and $\omega \cdot \xi=\widetilde{\omega} \cdot \xi=0$. Choosing now $\ama_1^*(x)=-\omega  \cdot x$, $x \in \Omega$,
then by using Lemma \ref{L4.14} , we obtain 
\begin{align*}
    \abs{ \widehat{\vartheta}_B (\xi)}
\les   \langle \xi\rangle^{2m} (e^{C/h^\ell} \|\Lambda^{(1)}-\Lambda^{(2)}\|^{\theta_3}+h^{\theta_4}),
\end{align*}
for some positive constant $C$. Here $\ell=\max(\ell_1, \ell_2) \gg 1$, $\theta_3=\min(\mu_4, \mu_5)\in(0,1)$ and $\theta_4=\min(\kappa_4, \kappa_5)$ $\in (0,1)$ depends only on $n$ and $m$.\\
Therefore, since ${\vartheta}_B$ satisfies \eqref{vartheta_B1} then by using Lemma \ref{LA}, we get our desired estimate.
\end{proof}


We are now in position to prove the stability estimate of the first order coefficient $B$ in the case $m >2$.

\begin{proof}[Proof of Theorem \ref{T1.3} (Case $m >2$)]
Using Lemmas \ref{L4.14}, \ref{L4.15} and the estimate \eqref{B1} we get from \eqref{B} 
\begin{align}\label{norm_B}
\|B \|_{L^{\infty}(\Omega)} 
&\les e^{C/h^\ell} \|\Lambda^{(1)}-\Lambda^{(2)}\|^{\theta_5}+h^{\theta_6},
\end{align}
for some $C >0$. Here $\ell=\max(\ell_2, \ell_3) \gg 1$, $\theta_5=\operatorname{min}(\mu_5, \mu_6) \in (0,1)$ and $\theta_6=\operatorname{min}(\kappa_5, \kappa_6) \in (0,1)$ depends on $n$ and $m$.\\
Let $0 < h_0 <1$, the above statement is valid for all $h < h_0$ and $\|\Lambda^{(1)}-\Lambda^{(2)}\|$ small enough. Then, if  $\|\Lambda^{(1)}-\Lambda^{(2)}\|< \varepsilon_0$, such that $-\frac{\theta_5}{2}\ln \varepsilon_0 \geq  {C/h_0^\ell}$, the taking $h=\bigr(-\frac{1}{C}\ln (\|\Lambda^{(1)}-\Lambda^{(2)}\|^{\frac{\theta_5}{2}})\bigr)^{\frac{-1}{\ell}}$ in \eqref{norm_B} implies
\begin{align}\label{4.22}
    \|B\|_{L^{\infty}(\Omega)} \les   \|\Lambda^{(1)}-\Lambda^{(2)}\|^{\frac{\theta_5}{2}}+  \abs{\ln \|\Lambda^{(1)}-\Lambda^{(2)}\|}^\frac{-\theta_6}{\ell}.
\end{align}
Moreover, with this choice of $h$ we have $0 <h < h_0$.\\
We also observe that this type of inequality holds if $\|\Lambda^{(1)}-\Lambda^{(2)}\|\geq \varepsilon_0$. Since in that case we can simply write
\begin{align}
      \|B\|_{L^{\infty}(\Omega)} \les M \les \frac{M}{(\sqrt{\varepsilon_0})^{\theta_5}}\|\Lambda^{(1)}-\Lambda^{(2)}\|^{\frac{\theta_5}{2}}.
\end{align}
This concludes the proof of stability estimate of first order coefficient $B$ in the case $m >2$.
\end{proof}

\subsection{Stability estimate for the symmetric tensor (first case \texorpdfstring{$m>2$}{Lg}) }

The aim of this section is to prove the stable determination of the symmetric tensor $A$ in the case $m>2$. To do this we need to estimate the function $\vartheta_A$, given by \eqref{A}.

\begin{lemma}\label{L4.17}
Let $\vartheta_A$ given by \eqref{A}. There exist $\ell_4 \gg 1$ and $\mu_7, \kappa_7 \in (0,1)$ such that we have
\begin{align}\label{4.54}
   \|\vartheta_A\|_{L^{\infty}(\Omega)}&\leq e^{C_7/{h^{\ell_4}}} \|\Lambda^{(1)}-\Lambda^{(2)}\|^{\mu_7} +h^{\kappa_7},
\end{align}
for any $h >0$ small enough. Here $C_7$ is a positive constant independent of $h$ and $\kappa_7, \ell_4$ depend on $n$ and $m$.

\end{lemma}

 \begin{proof}
 Let $\xi \in \mathbb{R}^n$. We fix $\omega, \widetilde{\omega} \in \mathbb{S}^{n-1}$ such that $\xi$, $\omega$ and $\widetilde{\omega}$ be three mutually orthogonal vectors in $\mathbb{R}^n$. Let us choose $\ama_1(x)=\omega \cdot x$, $x \in \Omega$ and $\ama_1^* (x)=e^{i x \cdot \xi}$. It is clear that $\ama_1 $ and $\ama_1^*$ satisfy \eqref{3.13} and \eqref{3.14}, respectively. Using now Lemma \ref{L4.12}, we obtain 
\begin{align*}
    &\abs{\widehat{ \vartheta}_A(\xi) }
    \les \langle \xi\rangle^{2m}\bigr(e^{C_4/{h^{\ell_1}}}\norm{\Lambda^{(1)}-\Lambda^{(2)}}^{\mu_4}+h^{\kappa_4}+\|B\|_{L^\infty(\Omega)}\bigr).
\end{align*}
Therefore, \eqref{norm_B} implies 
\begin{align*}
    &\abs{\widehat{ \vartheta}_A(\xi) }
    \les \langle \xi\rangle^{2m}(e^{C^\prime/h^{\ell^\prime}}\norm{\Lambda^{(1)}-\Lambda^{(2)}}^{\theta_7}+h^{\theta_8}),
\end{align*}
for some $C^\prime >0$. Here $\ell^\prime=\max(\ell_1, \ell)$, $\theta_7=\min(\mu_4, \theta_5)$ and $\theta_8=\min(\kappa_4, \theta_6)$.\\
Hence, since $\vartheta_A$ satisfies \eqref{A1} then by using Lemma \ref{LA} we deduce \eqref{4.54}. This ends the proof.
\end{proof}


\begin{proof}[Proof of Theorem \ref{T1.2} (Case $m >2$)]
By using Lemmas \ref{L4.5}, \ref{L4.11} and \ref{L4.17}, we get from the equality \eqref{A}
\begin{align}\label{normA}
  \|A\|_{L^{\infty}(\Omega)}&\leq e^{C/{h^{\ell_4}}} \|\Lambda^{(1)}-\Lambda^{(2)}\|^{\theta_9} +h^{\theta_{10}},   
\end{align}
for some positive constant $C$. Here $\theta_9=\min(\mu_1, \mu_3, \mu_7)$ and $\theta_{10}=\min(\kappa_1, \kappa_3, \kappa_7)$ depends only on $n$ and $m$.\\
Choosing now $h=\bigr(-\frac{1}{C}\ln (\|\Lambda^{(1)}-\Lambda^{(2)}\|^{\frac{\theta_9}{2}})\bigr)^{\frac{-1}{\ell_4}}$ and proceeding exactly as the last part of proof of Theorem \ref{T1.3} (Case $m >2$)
we deduce the stability for the second order perturbation in the case $m >2$.
\end{proof}

\subsection{Stability estimate for the electric potential (first case \texorpdfstring{$m>2$}{Lg}) }

We will now establish the stability estimate for the zeroth order term. Let us first estimate the Fourier transform of zeroth-order perturbations as follow.


\begin{lemma}\label{L3.12}
Let $q \in \mathcal{Q}(M)$. For all $\xi \in \mathbb{R}^n$, there exist $\ell_5 \gg 1$ and $\mu_8, \kappa_8 \in (0,1)$ such that the following estimate
\begin{align}\label{3.63}
  |\widehat{q}(\xi)| \les \langle \xi\rangle^{2m}   \bigr(e^{C_8/{h^{\ell_5}} }\|\Lambda^{(1)}-\Lambda^{(2)}\|^{\mu_8}+h^{\kappa_8}\bigr),
\end{align}
holds for any $h >0$ small enough. Here $C_8$ is a positive constant independent of $h$ and $\kappa_8, \ell_5$ depend on $n$ and $m$.
\end{lemma}

\begin{proof}
Substituting $u$ and $u^*$ into the left hand side of \eqref{3.6}. Then the third term on the left hand side of \eqref{3.6} becomes
\begin{align*}
   \int_{\Omega} q(x) u(x) \overline{u^*}(x)~dx=  \int_{\Omega} q(x) \ama_1(x) \overline{\ama_1^*}(x)~dx + \mathcal{J}_1(h),
\end{align*}
where 
\begin{align}
 \mathcal{J}_1(h)=   \int_{\Omega} q(x) (h\ama_2+r) \overline{(\ama_1^*+h\ama_2^*+r^*)}~dx+\int_{\Omega} q(x) \ama_1 \overline{(h\ama_2^*+r^*)}~dx.
\end{align}
We deduce from the Cauchy-Schwartz inequality and from the inequalities (\ref{3.12}), \eqref{3.12'}, \eqref{2.50} and \eqref{2.53} that
\begin{align}\label{3.61}
    \abs{ \mathcal{J}_1(h)}\les hN_{2m, \varrho}(\ama_1)N_{2m, \overline{\varrho }}(\ama_1^*).
\end{align}
In the same manner, we get
\begin{align}\label{3.62}
 &\int_{\Omega} (AD\cdot D u+ B\cdot Du) \overline{u^*}~dx  \cr
 &= \int_{\Omega}\bigr((-h^{-2}A\varrho\cdot \varrho-\frac{2i}{h}A\varrho \cdot D+AD\cdot D-\frac{i}{h} \varrho \cdot B +B \cdot D
 )(\ama_1+h\ama_2+r)\bigr)\overline{(\ama_1^*+h\ama_2^*+r^*)}dx\cr
 &:=\mathcal{J}_2(h),
\end{align}
where $\mathcal{J}_2(h)$ satisfies
\begin{align*}
  \abs{\mathcal{J}_2(h)}\les \bigr(h^{-2}\|A\|_{L^{\infty}}+ h^{-1}\|B\|_{L^{\infty}}\bigr)N_{2m, \varrho}(\ama_1)N_{2m, \overline{\varrho }}(\ama_1^*).  
\end{align*}
Therefore, using the estimates obtained above and the equality \eqref{3.6}, we deduce 
\begin{align}\label{3.59}
 \abs{ \int_{\Omega} q(x) \ama_1(x) \overline{\ama_1^*}(x)~dx} 
&\les  \bigr(h^{-2}\|A\|_{L^{\infty}(\Omega)}+h^{-1}\|B\|_{L^{\infty}(\Omega)}\cr &\qquad+e^{C/h }\|\Lambda^{(1)}-\Lambda^{(2)}\| +h\bigr) N_{2m, \varrho}(\ama_1)N_{2m, \overline{\varrho }}(\ama_1^*),
\end{align}
Let now $\xi \in \mathbb{R}^n$. We fix $\omega, \widetilde{\omega} \in \mathbb{S}^{n-1}$ such that $\xi$, $\omega$ and $\widetilde{\omega}$ be three mutually orthogonal vectors in $\mathbb{R}^n$. Choosing $\ama_1(x)=e^{-ix \cdot \xi}$ and $\ama_1^*(x)=1$ in \eqref{3.59} we get
\begin{align}\label{3.64}
  \abs{\widehat{q}(\xi)} \les \langle \xi\rangle^{2m}  \bigr(h^{-2}\|A\|_{L^{\infty}(\Omega)}+h^{-1}\|B\|_{L^{\infty}(\Omega)}+e^{C/h }\|\Lambda^{(1)}-\Lambda^{(2)}\|+h\bigr) .
\end{align}
This along with \eqref{norm_B}, \eqref{normA} and \eqref{1*} show that 
\begin{align*}
 \abs{\widehat{q}(\xi)} &\les \langle \xi\rangle^{2m} \bigr(e^{C_8/{h^{\widetilde{\ell}_4}} }\|\Lambda^{(1)}-\Lambda^{(2)}\|^{\theta_9}+h^{\widetilde{\theta}_{10}}+e^{C_8/{h^{\widetilde{\ell}}} }\|\Lambda^{(1)}-\Lambda^{(2)}\|^{\theta_5}+h^{\widetilde{\theta}_{6}}\bigr),
\end{align*}
for some $C_8 >0$ is a positive constant independent of $h$. Here $\widetilde{\ell}, \widetilde{\ell}_4 \gg 1$ and $\widetilde{\theta}_{6}, \widetilde{\theta}_{10} \in (0,1)$ depend on $n$ and $m$. Then, by choosing $\ell_5=\max(\widetilde{\ell}_4, \widetilde{\ell}) $, $\kappa_8=\min(\widetilde{\theta}_{10}, \widetilde{\theta}_{6})$ and $\mu_8=\min(\theta_5, \theta_9)$ we get our desired estimate.
\end{proof}

With the help of the above Lemma, we may now prove the stability estimate of the potential.

\begin{proof}[Proof of Theorem \ref{T1.4} (Case $m>2$)]

Let $R >1$, to be chosen later, we get from Lemma \ref{L3.12}
\begin{align}
\|q \|_{H^{-1}(\Omega)}^{2} &\leq \|q \|_{H^{-1}\left(\mathbb{R}^{n}\right)}^{2}=\int_{\langle \xi\rangle \leqslant R}\abs{\widehat{q}(\xi)}^{2}\langle \xi\rangle^{-2} d \xi+\int_{\langle \xi\rangle> R}\abs{\widehat{q}(\xi)}^{2}\langle \xi\rangle^{-2} d \xi \cr 
&\les R^{n+4m-2}\bigr(e^{{2C_8}/{h^{\ell_5}} }\|\Lambda^{(1)}-\Lambda^{(2)}\|^{2\mu_8}+h^{2\kappa_8}\bigr)+ M^2 R^{-2}.\label{3.650}
\end{align}
Choosing $R=h^\frac{-2\kappa_8}{n+4m}$ then, for $h_0 >0$ sufficiently small we obtain
\begin{align}
 \|q\|_{H^{-1}(\Omega)} \les e{^{C^\prime/h^{\ell_5}}}\|\Lambda^{(1)}-\Lambda^{(2)}\|^{\mu_8}+  h^\frac{2 \kappa_8}{n+4m},\label{3.660}
\end{align}
for some positive constant $C^{'}$ and all $h<h_0$.\\
Let $\varepsilon_0 >0$ such that $- \frac{\mu_8}{2}\ln \varepsilon_0 \geq  {C^\prime/h_0^{\ell_5}}$. If $\|\Lambda^{(1)}-\Lambda^{(2)}\|< \varepsilon_0$, we choose  $h=\bigr(-\frac{1}{C^{'}}\ln( \|\Lambda^{(1)}-\Lambda^{(2)}\|^{\frac{\mu_8}{2}})\bigr)^{-\frac{1}{\ell_5}}$ in \eqref{3.660}, thus, we obtain $0 < h \leq h_0$ and 
\begin{align}
     \|q\|_{H^{-1}(\Omega)} \les   \|\Lambda^{(1)}-\Lambda^{(2)}\|^{\frac{\mu_8}{2}}+  \abs{\ln \|\Lambda^{(1)}-\Lambda^{(2)}\|}^\frac{-2 \kappa_8}{(n+4m)\ell_5}.
\end{align}
Now if $\|\Lambda^{(1)}-\Lambda^{(2)}\|\geq \varepsilon_0$, we also observe that this type of inequality holds. Indeed, we can simply write
\begin{align}
     \|q \|_{H^{-1}(\Omega)} \les  M \les \frac{M}{\sqrt{\varepsilon_0}^{\mu_8}}\|\Lambda^{(1)}-\Lambda^{(2)}\|^{\frac{\mu_8}{2}}.
\end{align}
This concludes the proof for stability of zeroth order coefficient, in the case $m
>2$.
\end{proof}

\section{Stability estimates (second case \texorpdfstring{$m=2$}{Lg})} \label{section5}
The aim of this section is to proving the stable determination of the isotropic tensor $A=a ~id$, the vector field $B$ and the electric potential $q$ from the Dirichlet to Neumann map $\Lambda_{A,B,q}$, in the case $m=2$. The proof follows from similar arguments we used in the previous section.\smallskip\\
We consider for $j=1,2,$ $A^{(j)}=a^{(j)}\operatorname{id} \in \mathcal{A}_{\sigma_1}(M) \cap \mathcal{E}^\prime(\Omega)$, $B^{(j)} \in \mathcal{B}_{\sigma_2}(M) \cap \mathcal{E}^\prime(\Omega)$ and $q^{(j)} \in \mathcal{Q}(M)$ pair of admissible coefficients. We denote
\begin{align}
    a=a^{(2)}-a^{(1)},\quad B=B^{(2)}-B^{(1)}\quad\text{and  }~~~~ q=q^{(2)}-q^{(1)}.
\end{align}

\subsection{Stability estimate for the vector field (second case \texorpdfstring{$m=2$}{Lg})}

We derive in this section stability estimate for the vector field $B$ in the case $m=2$. As section \ref{section4}, we use the Hodge decomposition for a vector field $B$ and we write
\begin{align}
   B&=B^\prime+\nabla \vartheta_B,\label{6.0}  
\end{align}
where the vector field $B^{\prime} \in L^{ \infty}(\Omega, \mathbb{C}^n)$ satisfies \eqref{4.100} and 
\begin{align}\label{6.1}
    \|B^\prime\|_{L^{\infty}(\Omega)} \les \|\curl~B\|_{L^{\infty}(\Omega)}.
\end{align}
Moreover, the function $\vartheta_B \in W^{1, \infty}(\Omega, \C)$ satisfies 
\begin{align}\label{6.2}
     \| \vartheta_B\|_{H^{\sigma_2}(\Omega)} \les \|B\|_{H^{\sigma_2}(\Omega)},
\end{align}
here $\sigma_2> \frac{n}{2}+1$.\medskip\\
We start by estimate the $L^{\infty}-$norm of the vector field $B^\prime$.
\begin{lemma}\label{'T3.5}
Let $B^\prime$ given by \eqref{6.0}. Then there exist  $\mu_9, \kappa_9 \in (0, 1)$ such that we have the following estimate
  \begin{align}
  \|B^\prime\|_{L^{\infty}(\Omega)} \les e^{C_9/{h}} \|\Lambda^{(1)}-\Lambda^{(2)}\|^{\mu_9} +h^{\kappa_9},
      \label{'3.27}
  \end{align}
for any $h >0$ small enough. Here $C_9$ is a positive constant independent of $h$ and $\kappa_9 $ depends only on $n$.
\end{lemma}

\begin{proof}
Let $\xi \in \mathbb{R}^n$, choosing $\omega, \widetilde{\omega} \in \mathbb{S}^{n-1}$ such that $\xi, \omega$ and $\widetilde{\omega}$ be three mutually orthogonal vectors in $\R^n$. Let us choose $\ama_1 (x)=e^{-i x \cdot \xi}$ and $\ama_1^*(x)=1$, $x \in \Omega$ being a solutions of \eqref{3.13} and \eqref{3.14}, respectively. Using now Lemma \ref{L3.3} and the fact that $\varrho \cdot \xi=0$, we obtain
\begin{equation}
\abs{\varrho \cdot \widehat{B}(\xi)}\les \langle \xi\rangle^{4}(e^{C/h}\norm{\Lambda^{(1)}-\Lambda^{(2)}}+h).\label{'3.24}
\end{equation}
Doing the same analysis as the first part of proof of Lemma \ref{L4.7}, we get
\begin{align*}
     \|\curl~B\|_{L^{\infty}(\Omega)} \les e^{C_9/{h}} \|\Lambda^{(1)}-\Lambda^{(2)}\|^{\mu_9} +h^{\kappa_9},
\end{align*}
for some positive constant $C_9$ and $\mu_9, \kappa_9 \in (0,1)$. Now, by using \eqref{6.1} we get our desired estimate.
\end{proof}
With the help of above Lemma, we may estimate $\vartheta_B$ given by \eqref{6.0}. 
\begin{lemma}
Let $\vartheta_B$ given by \eqref{6.0}. There exist $\mu_{10}, \kappa_{10} \in (0, 1)$ such that we have
\begin{align*}
 \|\vartheta_B \|_{W^{1, \infty}(\Omega)} \les    e^{C_{10}/{h}} \|\Lambda^{(1)}-\Lambda^{(2)}\|^{\mu_{10}} +h^{\kappa_{10}}.
\end{align*}
Here $C_{10}$ is a positive constant independent of $h$ and $\kappa_{10} \in (0,1)$ depends only on $n$.
\label{'P3.7}
\end{lemma}

\begin{proof}
Let $\xi \in \mathbb{R}^{n}$. We choose $\omega, \widetilde{\omega} \in \mathbb{S}^{n-1}$ such that $\xi, \omega$ and $\widetilde{\omega}$ be three mutually orthogonal vectors in $\R^{n}$. Substituting $B$  into the left hand side of \eqref{3.25}, we get for $\ama_1(x)=e^{-i x \cdot \xi}$
\begin{equation}\label{'3.33}
\int_\Omega e^{-i x \cdot \xi} (\varrho \cdot \nabla \vartheta_B) \overline{\ama_1^*} ~dx=-\int_\Omega e^{-i x \cdot \xi} (\varrho \cdot B^{'} )\overline{\ama_1^*}~dx + \mathcal{I}(h),
\end{equation}
where $\mathcal{I}$ satisfies 
\begin{align}\label{'3.34}
  |  \mathcal{I}(h)| \les  \langle \xi\rangle^4(e^{C/h} \|\Lambda^{(1)}-\Lambda^{(2)}\|+h)N_{4,  \overline{\varrho}}(\ama_1^*).
\end{align}
We apply now an integration by parts for the left hand side of the above equality, we obtain
\begin{align*}
- \int_{\Omega} e^{-i x \cdot \xi}  \vartheta_B  (\varrho \cdot \nabla \overline{ \ama_1^*})\mathrm{~d} x =- \int_{\Omega} e^{-i x \cdot \xi}(\varrho \cdot B^\prime)  \overline{\ama_1^*}\mathrm{~d} x+ \mathcal{I}(h),  
\end{align*}
here we used the fact that $\vartheta_B =0$, on $\Gamma $ and $\varrho \cdot \xi=0$.\\
Therefore, we obtain for $\ama_1^*(x)=-\omega  \cdot x$, $x \in \Omega$, that
\begin{align*}
    \abs{ \widehat{\vartheta}_B(\xi)}
\les   \langle \xi\rangle^4 \bigr(\|B^\prime \|_{L^{\infty}(\Omega)}+e^{C/h}\|\Lambda^{(1)}-\Lambda^{(2)}\|+h \bigr).
\end{align*}
This inequality and Lemma \ref{'T3.5} allow us to obtain for some constant $C^\prime >0$ that
\begin{align*}
 \abs{ \widehat{\vartheta}_B(\xi)} &\les  \langle \xi\rangle^4 (e^{C^\prime/h}\|\Lambda^{(1)}-\Lambda^{(2)}\|^{\mu_9}+h^{\kappa_9}),
\end{align*}
here $\kappa_9 \in (0,1)$ and $\mu_9=\mu_9(n) \in (0,1)$.\\
In the other hand, since ${\vartheta}_B$ satisfies \eqref{6.2} then by using Lemma \ref{LA} we get our desired estimate.
\end{proof}

By using the above results, we are able to prove the stability estimate for first order coefficient $B$, in the case $m=2$. 

\begin{proof}[Proof of Theorem \ref{T1.3} (Case $m=2$)]
By Lemmas \ref{'T3.5} and \ref{'P3.7}, we get from \eqref{6.0}
\begin{align}
\|B \|_{L^{\infty}(\Omega)} &  \les  \|B^\prime\|_{L^{\infty}(\Omega)}+\|\vartheta_B \|_{W^{1,\infty}(\Omega)}\cr
&\les e^{C^\prime/h} \|\Lambda^{(1)}-\Lambda^{(2)}\|^{\theta_{13}} +h^{\theta_{14}},\label{6.3}
\end{align}
with $C^\prime$ is a positive constant independent of $h$, $\theta_{13}=\operatorname{min}(\mu_9, \mu_{10})$ and $\theta_{14}=\min(\kappa_9, \kappa_{10}) \in (0,1)$ depends only on $n$.\\
Let $\varepsilon_0 \in (0, 1)$ such that $-\ln \varepsilon_0 \geq  {C^\prime/h_0}$. Then if $\|\Lambda^{(1)}-\Lambda^{(2)}\|< \varepsilon_0$, choosing $h=\bigr(-\frac{1}{C^\prime}\ln (\|\Lambda^{(1)}-\Lambda^{(2)}\|^{\frac{\theta_{13}}{2}})\bigr)^{-1}$ in \eqref{6.3} and proceeding exactly as the last part of proof of Theorem \ref{T1.3} (Case $m >2$)
we conclude that Theorem \ref{T1.3} is completely proved.
\end{proof}

\subsection{Stability estimate for the isotropic tensor (second case \texorpdfstring{$m=2$}{Lg})}
We derive in this section stability estimate for the isotropic matrix $A=a ~id$. First, we will use Lemma \ref{L3.3} to estimate the Fourier transform of $a=a^{(2)}-a^{(1)}$.

\begin{lemma}\label{'L3.8}
Let $a$ as above. There exist $\mu_{11}, \kappa_{11} \in (0,1)$ such that  we have the following estimate
\begin{align*}
  \abs{\widehat{a}(\xi)}  \les \langle\xi \rangle^4\bigr(e^{C_{11}/{h} }\|\Lambda^{(1)}-\Lambda^{(2)}\|^{\mu_{11}}+h^{\kappa_{11}}\bigr),  
\end{align*}
for any $h > 0$ small enough and for all $\xi \in \mathbb{R}^n$. Here $C_{11}$ is a positive constant independent of $h$ and $\kappa_{11}$ depends only on $n$.
\end{lemma}

\begin{proof}
Let $\xi \in \mathbb{R}^n$. We fix $\omega, \widetilde{\omega} \in \mathbb{S}^{n-1}$ such that $\xi$, $\omega$ and $\widetilde{\omega}$ be three mutually orthogonal vectors in $\mathbb{R}^n$. Choosing $\ama_1 (x)=(x \cdot \omega) e^{-i x\cdot \xi}$ and $\ama_1^*(x)=1$,~~for $x \in \Omega$. It is clear that $\ama_1 $ and $\ama_1^*$ satisfy respectively \eqref{3.13} and \eqref{3.14}. Then using Lemma \ref{L3.3} we get
\begin{align*}
  \abs{\widehat{a}(\xi)}\les\langle \xi\rangle^4\bigr(\|B\|_{L^{\infty}(\Omega)}+ e^{C/h }\|\Lambda^{(1)}-\Lambda^{(2)}\|+h \bigr).  
\end{align*}
 Now using \eqref{6.3} we get our desired estimate with $C_{11}=\max(C^\prime, C)$, $\mu_{11}=\theta_{13}$ and $\kappa_{11}=\theta_{14}(n)$.
\end{proof}

With the help of the above Lemma, we may now prove the stability estimate of the second-order perturbation, in the case $m=2$.

\begin{proof}[Proof of Theorem \ref{T1.2} (Case $m=2$)]
By Lemma \ref{'L3.8} we have 
  \begin{align*}
  \abs{\widehat{a}(\xi)}  \les \langle\xi \rangle^4\bigr(e^{C_{11}/{h} }\|\Lambda^{(1)}-\Lambda^{(2)}\|^{\mu_{11}}+h^{\kappa_{11}}\bigr). 
\end{align*}
Then by using Lemma \ref{LA}, there exist $\theta_{15}, \theta_{16}=\theta_{16}(n) \in (0,1)$ such that we have
\begin{align}\label{normA1}
   \|a\|_{L^{\infty}(\Omega)}&\leq e^{C^\prime/h} \|\Lambda^{(1)}-\Lambda^{(2)}\|^{\theta_{15}} +h^{\theta_{16}},   
\end{align}
for some constant $C^\prime >0$.\\
Now choosing $h=\bigr(-\frac{1}{C^\prime}\ln (\|\Lambda^{(1)}-\Lambda^{(2)}\|^{\frac{\theta_{15}}{2}})\bigr)^{-1}$ and proceeding exactly as the last part of proof of Theorem \ref{T1.3} (Case $m >2$)
we get our desired estimate. This concludes the proof for stability of second order coefficients from the D-to-N map.
\end{proof}

\subsection{Stability estimate for the electric potential (second case \texorpdfstring{$m=2$}{Lg})}

Now, we establish the stability result for the zeroth order term $q$. Let us first establish the following estimate of the Fourier transform of $q$.


\begin{lemma}
There exist $\ell_6 > 1$ and $\kappa_{12}, \mu_{12} \in (0,1)$ such that we have
\begin{align}\label{'3.40}
  |\widehat{q}(\xi)| \les \langle \xi\rangle^4 \left( e^{C_{12}/{h^{\ell_6}} }\|\Lambda^{(1)}-\Lambda^{(2)}\|^{\mu_{12}}+h^{\kappa_{12}}\right),
\end{align}
for any $h >0$ small enough and for all $\xi \in \mathbb{R}^n$. Here $C_{12}$ is a positive constant independent of $h$ and $\kappa_{12} \in (0,1)$ depends on $n$.
\label{'L3.10}
\end{lemma}

\begin{proof}
Substituting $u$ and $u^*$ into the left hand side of \eqref{3.6}. Then the third term in the left hand side of \eqref{3.6} becomes
\begin{align*}
   \int_{\Omega} q(x) u(x) \overline{u^*}(x)~dx=  \int_{\Omega} q(x) \ama_1(x) \overline{\ama_1^*}(x)~dx + \mathcal{J}_1(h),
\end{align*}
where 
\begin{align}
 \mathcal{J}_1(h)=   \int_{\Omega} q(x) (h\ama_2+r) \overline{(\ama_1^*+h\ama_2^*+r^*)}~dx+\int_{\Omega} q(x) \ama_1 \overline{(h\ama_2^*+r^*)}~dx,
\end{align}
then by using the Cauchy-Schwartz inequality together with the inequalities (\ref{3.12}), \eqref{3.12'}, \eqref{2.50} and \eqref{2.53} we deduce 
\begin{align}\label{'3.37}
    \abs{ \mathcal{J}_1(h)}\les hN_{4,  {\varrho}}(\ama_1)N_{4,  \overline{\varrho}}(\ama_1^*).
\end{align}
Similarly, we have 
\begin{align}\label{'3.38}
&\int_{\Omega} (-a\Delta u+ B\cdot Du) \overline{u^*}~dx \cr
 &= \int_{\Omega}\big((\frac{-2}{h}a \varrho \cdot \nabla-a \Delta-\frac{i}{h} \varrho \cdot B +B \cdot D
)(\ama_1+h\ama_2+r) \big)\overline{(\ama_1^*+h\ama_2^*+r^*)}dx\cr&:= \mathcal{J}_2(h),
\end{align}
where $\mathcal{J}_2$ can be estimated as
\begin{align*}
  \abs{\mathcal{J}_2(h)}\les h^{-1}(\|a\|_{L^{\infty}}+ \|B\|_{L^{\infty}})N_{4,  {\varrho}}(\ama_1)N_{4,  \overline{\varrho}}(\ama_1^*).  
\end{align*}
Combining now the above inequalities, we obtain from \eqref{3.6}
\begin{align}\label{3.82}
 \abs{ \int_{\Omega} q(x) \ama_1(x) \overline{\ama_1^*}(x)~dx} 
 &\les (h^{-1}\|B\|_{L^{\infty}(\Omega)}+h^{-1}\|a\|_{L^{\infty}(\Omega)}\cr&\quad+e^{C/h }\|\Lambda^{(1)}-\Lambda^{(2)}\|+h)N_{4,  {\varrho}}(\ama_1) N_{4,  \overline{\varrho}}(\ama_1^*),
\end{align}
Let now $\xi \in \mathbb{R}^n$. We fix $\omega, \widetilde{\omega} \in \mathbb{S}^{n-1}$ such that $\xi$, $\omega$ and $\widetilde{\omega}$ be three mutually orthogonal vectors in $\mathbb{R}^n$. Choosing $\ama_1(x)=e^{-ix \cdot \xi}$ and $\ama_1^*(x)=1$ which respectively solve the transport equations \eqref{3.13} and \eqref{3.14}. Now by putting these particular choices of $\ama_1$ and $\ama_1^*$ in \eqref{3.82}, we get 
\begin{align}\label{'q}
   \abs{\widehat{q}(\xi)} \les \langle \xi\rangle^4 \bigr(h^{-1}\|B\|_{L^{\infty}(\Omega)}+h^{-1}\|a\|_{L^{\infty}(\Omega)}+e^{C/h }\|\Lambda^{(1)}-\Lambda^{(2)}\|+h\bigr).
\end{align}
From this and the inequalities \eqref{6.3} and \eqref{normA1}, we can deduce 
\begin{align}
   \abs{\widehat{q}(\xi)} \les \langle \xi\rangle^4 h^{-1}\bigr(e^{C/h }\|\Lambda^{(1)}-\Lambda^{(2)}\|^{\theta_{17}}+h^{\theta_{18}}\bigr),
\end{align}
for some positive constant $C$. Here $\theta_{17}=\min(\theta_{13},\theta_{15}) \in (0,1)$ and $\theta_{18}=\min(\theta_{14}, \theta_{16}) \in (0, 1)$ depends on $n$.
Using now \eqref{1*} we get our desired estimate.
\end{proof}

With the help of the above Lemma, we may now prove the stability estimate of the potential $q$, in the case $m=2$.

\begin{proof}[Proof of Theorem \ref{T1.4} (Case $m=2$)]
By Lemma \ref{'L3.8} we have 
  \begin{align*}
  |\widehat{q}(\xi)| \les \langle \xi\rangle^4 \left( e^{C_{12}/{h^{\ell_6}} }\|\Lambda^{(1)}-\Lambda^{(2)}\|^{\mu_{12}}+h^{\kappa_{12}}\right). 
\end{align*}
Then for $R >1$, to be chosen later, we obtain
\begin{align}
\|q \|_{H^{-1}(\Omega)}^{2} &\leq \|q \|_{H^{-1}\left(\mathbb{R}^{n}\right)}^{2}=\int_{\langle \xi\rangle \leqslant R}\abs{\widehat{q}(\xi)}^{2}\langle \xi\rangle^{-2} d \xi+\int_{\langle \xi\rangle> R}\abs{\widehat{q}(\xi)}^{2}\langle \xi\rangle^{-2} d \xi \cr 
&\les R^{8}\bigr(e^{{2C_{12}}/{h^{\ell_6}} }\|\Lambda^{(1)}-\Lambda^{(2)}\|^{2\mu_{12}}+h^{2\kappa_{12}}\bigr)+ R^{-2}.\label{3.6500}
\end{align}
Choosing $R=h^\frac{-2\kappa_{12}}{10}$ then, for $h_0 >0$ sufficiently small we obtain
\begin{align}
 \|q\|_{H^{-1}(\Omega)} \les e{^{C^\prime/h^{\ell_6}}}\|\Lambda^{(1)}-\Lambda^{(2)}\|^{\mu_{12}}+  h^\frac{2 \kappa_{12}}{10},\label{3.6600}
\end{align}
for some positive constant $C^{'}$ and all $h<h_0$.\\
Now choosing $h=\bigr(-\frac{1}{C^\prime}\ln (\|\Lambda^{(1)}-\Lambda^{(2)}\|^{\frac{\mu_{12}}{2}})\bigr)^{\frac{-1}{\ell_6}}$ and proceeding exactly as the last part of proof of Theorem \ref{T1.4} (Case $m >2$) we get our desired estimate. This concludes the proof for stability of zeroth order coefficient.
\end{proof}


\appendix 
\section{Proof of Lemma \ref{L4.1}}

\label{appendixA}
This appendix is devoted to the proof of Lemma \ref{L4.1}. We first establish the following result.
\begin{lemma}\label{Lann}
Let $F \in H^{k-1}(\Omega, \C^n)$. Then the boundary value problem
\begin{equation}\label{ann1}
\left\{\begin{array}{ll}
-\dive(\nabla_{\text{sym}}V)+\frac{1}{n} \nabla(\dive~ V) =F & \mathrm{in}\, \Omega,\cr
V  =0  &\mathrm{on}\, \Gamma,
\end{array}
\right.
\end{equation}
admits an unique solution $V \in H^{k+1}(\Omega, \C^n)$ which satisfy the estimate
\begin{align}\label{ann2}
    \|V\|_{H^{k+1}(\Omega)}  \les \|F\|_{H^{k-1}(\Omega)} .
\end{align}
\end{lemma}

\begin{proof}
Following Theorem $3.3$ in \cite{SV}, we shall show that the boundary value problem \eqref{ann1} is elliptic with zero kernel and zero co-kernel.\\ 
First, we start by checking the ellipticity of the operator 
\begin{align}
    P(D)=-\dive~\nabla_{\text{sym}}+\frac{1}{n} \nabla~\dive.
    \label{Ann.1}
\end{align}
The principal symbol of operator $P(D)$ is given by 
\begin{align}
    p(\xi)=\frac{\abs{\xi}^2}{2}\operatorname{id}+(\frac{1}{2}-\frac{1}{n}) \xi ~{}^t \xi,\quad \xi \in \R^n.
\end{align}
Then we get, for all $V \in C^\infty(\Omega, \C^n)$
\begin{align}\label{Ann.2}
  \big(p(\xi)V, V  \big)=\frac{\abs{\xi}^2}{2}\norm{V}^2+(\frac{1}{2}-\frac{1}{n}) \norm{\xi \cdot V}^2,
\end{align}
here $\big(\cdot, \cdot \big)$ denotes the inner product scalar in ${(L^2(\Omega ))^n}$ with associated norm $\norm{\cdot}$.\\
Both coefficients of the right hand side of \eqref{Ann.2} are real valued and positive for $\xi \neq 0$. This proves that the operator $P$ is elliptic. Actually, we have shown that $P$ is strongly elliptic operator. Then by applying Proposition 11.10 ( Section $5$ in \cite{taylor}), the Lopatinskii condition (Condition  \RomanNumeralCaps{3} in \cite{Volevich} ) is satisfied for the Dirichlet problem \eqref{ann1}. Hence, the ellipticity of the boundary value problem \eqref{ann1} is proved.\\ 
We move now to prove the triviality of the kernel of problem \eqref{ann1} i.e., the following homogeneous problem 
\begin{equation}\label{Ann.3}
\left\{\begin{array}{ll}
P(D)V =0 & \mathrm{in}\, \Omega,\cr
V  =0  &\mathrm{on}\, \Gamma,
\end{array}
\right.
\end{equation}
has only zero solution. For this purpose, we consider the equation 
\begin{align}\label{Ann.4}
   P(D)V =0 \quad \mathrm{in}\, \Omega.
\end{align}
Due to the ellipticity regularity, $V$ is smooth. Let $V \in C^\infty(\Omega, \C^n)$ such that $V  =0,~  \mathrm{on}\, \Gamma$. We are going to show that $V$ solves \eqref{Ann.4} if and only if 
\begin{align}\label{Ann.5}
 \nabla_{\text{sym}}V=g ~id,   
\end{align}
with $g$ is a scalar function.\\
If \eqref{Ann.5} holds then we have 
\begin{align*}
    \dive V = \sum_{k=1}^n (\nabla_{\text{sym}}V)_{kk}= n g \quad\textrm{and} \quad \dive(\nabla_{\text{sym}}V)= \nabla g.
\end{align*}

This immediately implies that $P(D)V=0$ in $\Omega$. Now let $V$ be a solution to the equation \eqref{Ann.4}. After integrating by parts, we get 
\begin{align*}
 \big(\nabla_{\text{sym}}V, \nabla_{\text{sym}}V  \big)_{(L^2(\Omega ))^{n^2}}=   - \big(V, \dive(\nabla_{\text{sym}}V)  \big).
\end{align*}
Using now the fact that $V$ satisfies \eqref{Ann.4} and doing again an integration by parts we end up with
\begin{align}\label{Ann.6}
 \big(\nabla_{\text{sym}}V, \nabla_{\text{sym}}V  \big)_{(L^2(\Omega ))^{n^2}}=  \frac{1}{n} \big(\dive V, \dive V  \big).
\end{align}
We can simply write the symmetric tensor $\nabla_{\text{sym}}V$ as the sum 
\begin{align}\label{Ann.7}
    \nabla_{\text{sym}}V = M + g_0 ~id,
\end{align}
where $\operatorname{trace} M =0$ and $g_0$ a scalar function.
Then we get
\begin{align}\label{Ann.9}
    \big(\nabla_{\text{sym}}V, \nabla_{\text{sym}}V  \big)_{(L^2(\Omega ))^{n^2}}=\norm{M}^2 + n \norm{g_0}^2,
\end{align}
here we used the fact that $\operatorname{trace} M =0$ then we get 
$ \big(M, g_0 ~id  \big)_{(L^2(\Omega ))^{n^2}}=0$. Now, by taking the trace of equation \eqref{Ann.7} we obtain 
\begin{align}\label{Ann.8}
    \dive V= n g_0.
\end{align}
So, by substituting \eqref{Ann.8} and \eqref{Ann.9} in \eqref{Ann.6} we find out
\begin{align}
    n \norm{g_0}^2=\norm{M}^2 + n \norm{g_0}^2, \quad \text{with}~~n \geq 3.
\end{align}
Then $M \equiv 0$. This entails that $V$ satisfy the equation \eqref{Ann.5}. Hence since $V  =0,~~ \mathrm{on}\, \Gamma$, then by Theorem 1.3 in \cite{DS} we get $V \equiv 0$. Thus the boundary value problem \eqref{Ann.3} has zero kernel. \\
Next, we are going to prove that the problem \eqref{ann1} has the trivial co-kernel. Let $W \in C^\infty(\Omega, \C^n)$ be orthogonal
to the range of operator $P$ i.e.,
\begin{align}\label{Ann.10}
  \big(W, PV  \big)=0,
\end{align}
for all $V \in C^\infty(\Omega, \C^n)$ such that $V  =0,~~ \mathrm{on}\, \Gamma$. Our goal is to show that $W \equiv 0$. For this purpose, let $V \in C_0^\infty(\Omega)$ then by integrating by part we get 
\begin{align*}
  \big(W, PV  \big)=\big(PW, V  \big), ~~\forall V \in C_0^\infty(\Omega).
\end{align*}
This and \eqref{Ann.10} imply that 
\begin{align}\label{Ann.11}
    PW \equiv 0.
\end{align}
Let now $U \in C^\infty(\Gamma, \C^n)$ be arbitrary, by Lemma 5.4 in \cite{DS} we have the existence of $V \in C^\infty(\Omega)$ such that
\begin{align}\label{Ann.11'}
\left\{\begin{array}{ll}
V_{|_\Gamma}=0, \cr
\bigr(\frac{1}{n}  \dive( V)\nu_i-\sum_{j=1}^n \nu_j(\nabla_{\text{sym}}V)_{ij} \bigr)_{|_\Gamma}= U_i,\quad \forall i=1, \ldots, n.
\end{array}
\right.
\end{align}
Here $\nu$ denotes the unit outer normal vector to the boundary $\Gamma$. By using Green's formula we get
\begin{align*}
  \big(W, PV  \big)&=\big(PW, V  \big)+\big(W, U  \big)_0\cr
 &=\big(W, U  \big)_0.
\end{align*}
Here $\big( \cdot, \cdot \big)_0$ denotes the inner product scalar in ${(L^2(\Gamma ))^{n}}$. Then by \eqref{Ann.10} we obtain
\begin{align*}
    \big(W, U  \big)_0=0,\quad \forall U \in C^\infty(\Gamma).
\end{align*}
Which implies that 
\begin{align}\label{Ann.12}
   W=0,  \quad \text{on} ~~\Gamma.
\end{align}
From \eqref{Ann.11} and \eqref{Ann.12} we deduce that $W$ solves the homogeneous boundary value problem \eqref{Ann.3} then we obtain $W \equiv 0$. Finally, applying the standard theorem on normal solvability we get our desired results.
\end{proof}

\begin{proof}[Proof of Lemma \ref{L4.1}]
Let us first assume that \eqref{4.1}, \eqref{4.2} and \eqref{4.3} are valid. By taking the trace of \eqref{4.1} and using the fact that $A^\prime$ satisfies \eqref{4.3} we get
\begin{align*}
     \operatorname{trace}(A)= \operatorname{trace}(\nabla_{\text{sym}}V)+n \vartheta_A=\dive~V+n \vartheta_A, 
\end{align*}
which implies
\begin{align}\label{ann03}
  \vartheta_A=\frac{1}{n}\bigr(\operatorname{trace}(A)-\dive~V\bigr),
\end{align}
and 
\begin{align}\label{ann3}
 \nabla \vartheta_A=\frac{1}{n}\Big(\nabla\big(\operatorname{trace}(A))-\nabla(\dive~V)\Big).
\end{align}
Moreover, by taking the divergence of \eqref{4.1} and using again the fact that $A^\prime$ satisfies \eqref{4.3} we get
\begin{align}\label{ann4}
     \dive~A= \dive~(\nabla_{\text{sym}}V)+\nabla \vartheta_A.
\end{align}
By substituting \eqref{ann3} in \eqref{ann4}, it easy to see that $V$ satisfies the following boundary value problem
\begin{equation}\label{ann5}
\left\{\begin{array}{ll}
-\dive(\nabla_{\text{sym}}V)+\frac{1}{n} \nabla(\dive~ V) =  -\dive~A+\frac{1}{n}\nabla\bigr(\operatorname{trace}(A)\bigr)& \mathrm{in}\, \Omega,\cr
V  =0  &\mathrm{on}\, \Gamma.
\end{array}
\right.
\end{equation}
Conversely, since $A \in H^k(\Omega)$ then by using Lemma $\ref{Lann}$ for $F=-\dive~A+\frac{1}{n}\nabla\bigr(\operatorname{trace}(A)\bigr)$, there exist $V \in H^{k+1}(\Omega, \C^n)$ solving \eqref{ann5} and satisfy 
\begin{align}\label{ann6}
    \|V\|_{H^{k+1}(\Omega)}  \les \|A\|_{H^{k}(\Omega)} .
\end{align}
Hence, by putting 
\begin{align}\label{ann7}
  \vartheta_A=\frac{1}{n}(\operatorname{trace}(A)-\dive~V)\qquad \textrm{and} \quad   A^\prime=A- \nabla_{\text{sym}}V - \vartheta_A ~id,
\end{align}
then since $V$ solving \eqref{ann5}, we get
\begin{align*}
    \dive~ A^\prime&=\dive~A-\dive~( \nabla_{\text{sym}}V )- \nabla\vartheta_A \\
    &=\dive~A-\dive~( \nabla_{\text{sym}}V )- \frac{1}{n}\nabla\bigr(\operatorname{trace}(A)\bigr)+\frac{1}{n} \nabla(\dive~V)\\
    &=0.
\end{align*}
In addition, we get from \eqref{ann7}
\begin{align*}
  \operatorname{trace} A^\prime&=\operatorname{trace}(A)- \dive~V - n \vartheta_A\\
  &=\operatorname{trace}(A)- \dive~V-\bigr(\operatorname{trace}(A)-\dive~V\bigr)\\
  &=0.
\end{align*}
Moreover, since $V$ satisfies the estimate \eqref{ann6} we conclude from \eqref{ann7} that 
\begin{align}\label{ann8}
 \|\vartheta_A\|_{H^{k}(\Omega)}  \les \|A\|_{H^{k}(\Omega)}  \qquad \textrm{and} \quad  \|A^\prime\|_{H^{k}(\Omega)}  \les \|A\|_{H^{k}(\Omega)} .
\end{align}
This completes the proof.
\end{proof}


\nocite*
\bibliographystyle{plain}
\bibliography{sample.bib}

\end{document}